\documentclass[11pt,twoside]{amsart}
\usepackage{t1enc}
\usepackage[latin2]{inputenc}

\usepackage{amsmath}
\usepackage{amsthm}
\usepackage{amscd}
\usepackage{amssymb}
\usepackage[T1]{fontenc}

\usepackage{mathtools}

\usepackage{tikz}

\usepackage{wasysym}

\usepackage{verbatim}
\usepackage{amscd}

\usepackage{fancyhdr}

\usepackage{fancybox}
\usepackage{url}
\usepackage{color}

\usepackage{mathpazo} 

\newcommand{\mc}{\mathcal}
\newcommand{\mbb}{\mathbb}
\newcommand{\mrm}{\mathrm}
\newcommand{\mf}{\mathfrak}
\newcommand{\mbf}{\mathbf}

\newcommand{\0}{\emptyset}
\newcommand{\al}{\alpha}
\newcommand{\be}{\beta}
\newcommand{\ga}{\gamma}

\newcommand{\ka}{\kappa}
\newcommand{\lam}{\lambda}
\newcommand{\eps}{\varepsilon}
\newcommand{\de}{\delta}
\newcommand{\om}{\omega}

\newcommand{\wt}{\widetilde}

\newcommand{\dom}{\mathrm{dom}}
\newcommand{\ran}{\mathrm{ran}}

\newcommand{\bc}{\begin{center}}
\newcommand{\ec}{\end{center}}

\newcommand{\clrest}{\!\upharpoonright\!}

\newtheorem{claim}{Claim}

\newtheorem{thm}{Theorem}[section]

\newtheorem{prop}[thm]{Proposition}
\newtheorem{cor}[thm]{Corollary}
\newtheorem{fact}[thm]{Fact}

\theoremstyle{definition}

\newtheorem{prob}[thm]{Problem}

\newtheorem{exa}[thm]{Example}

\newtheorem{rem}[thm]{Remark}

\title{Analytic P-ideals and Banach spaces}

\author{Piotr Borodulin--Nadzieja}
\address[Piotr Borodulin-Nadzieja]{Instytut Matematyczny, Uniwersytet Wroc\l awski, \ \ \  pl. Grunwaldzki 2/4, 50-384 Wroc\l aw, Poland}
\email{pborod@math.uni.wroc.pl}

\author{Barnab\'as Farkas}
\address[Barnab\'as Farkas]{TU Wien, Vienna, Austria}
\email{barnabasfarkas@gmail.com}

\thanks{The first author was supported by
 National Science Center project no. 2018/29/B/ST1/00223. The second author was supported by the Austrian Science Fund (FWF) project no. P29907.}


\begin{document}

\maketitle

\begin{abstract} 
We study the interplay between Banach space theory and theory of analytic P-ideals.
Applying the observation that, up to isomorphism, all Banach spaces
with unconditional bases can be constructed in a way very similar to the
construction of analytic P-ideals from submeasures, we point out numerous
symmetries between the two theories. Also, we investigate a special case,
 the interactions between combinatorics of families of finite sets, topological properties of the ``Schreier type'' Banach spaces associated to these families, and the complexity of ideals generated by the canonical bases in these
spaces. Among other results, we present some new examples of Banach spaces and analytic P-ideals, a new characterization of precompact families and its applications to enhance Pt\'ak's Lemma and Mazur's Lemma.
\end{abstract}

\section{Introduction}
Certain aspects of the interplay between Banach space theory and the theory of analytic P-ideals have already been studied in the past two decades. For example, Veli\v{c}kovi\'{c} and Louveau (\cite{LV}) considered several ideals inspired by classical Banach spaces; Veli\v{c}kovi\'{c}
(\cite{veli-tsirelson}) and Farah (\cite{farah-tsirelson}) used the Tsirelson space, a classical construction in Banach space theory, to construct certain peculiar analytic P-ideals disproving a conjecture proposed by Kechris and Mazur; on the other hand, Drewnowski and Labuda (\cite{Drewnowski10} and \cite{Drewnowski17}) studied analytic P-ideals from the point of view of specialists in Banach space theory, obtaining certain general results.

In \cite{Our-paper} we proposed a more systematic study of the connections between these ideals and Banach spaces. The main tool for our studies was representability of ideals in Banach spaces, a natural generalization of the concept of summable
ideals. We say that an ideal on the set $\om$ of all natural numbers is \emph{representable} in a Banach space $X$ if it is of the form \[ \mc{I}^X_{(x_n)}=\bigg\{A\subseteq \omega\colon \sum_{n\in A} x_n\;\text{is unconditionally
convergent in}\;X\bigg\}\]
for some sequence $(x_n)$ from $X$. In Section \ref{reprsetting} we give a brief overview of this notion.

\medskip
In this article we will present a way to connect analytic P-ideals and Banach spaces in a tighter way. It is well-known that each analytic P-ideal is the exhaustive ideal $\mrm{Exh}(\varphi)=\{A\subseteq\om\colon \varphi(A\setminus n)\to 0\}$ of some lower-semicontinuous (lsc) submeasure $\varphi$ on $\om$, and for $F_\sigma$ P-ideals we
can find such submeasures for which $\mrm{Exh}(\varphi)$ coincides with the ``finite'' ideal $\mrm{Fin}(\varphi)=\{A\subseteq\om\colon \varphi(A)<\infty\}$ of $\varphi$ (see Section \ref{analPsec} for a detailed introduction to analytic P-ideals).
Considering an extended norm $\Phi$ on $\mathbb{R}^\om$ (that is, a ``norm'' which may attain infinite values) it makes sense to define the  \emph{exhaustive space} $\mrm{EXH}(\Phi)=\{x\in\mbb{R}^\om\colon \Phi(P_{\om\setminus n}(x))\to 0\}$ where
$P_A\colon \mbb{R}^\om\to\mbb{R}^\om$ is the natural projection associated to $A\subseteq\om$, and the \emph{finite space} $\mrm{FIN}(\Phi)=\{x\in\mbb{R}^\om\colon \Phi(x)<\infty\}$ of $\Phi$ in the same manner. Under some natural assumptions on $\Phi$, both
$\mrm{EXH}(\Phi)$ and $\mrm{FIN}(\Phi)$ equipped with the restrictions of $\Phi$ are Banach spaces. Moreover, it turns out that each Banach space with unconditional basis is (up to isomorphism) an exhaustive space for some extended norm.

In Section \ref{main}, first of all we show that if an ideal is representable in a Banach space, then it is representable in $\mrm{EXH}(\Phi)$ for some $\Phi$ via its canonical (unconditional) basis $(e_n)$. Pointing out the interactions between analytic P-ideals and these spaces, we present numerous statements equivalent to (i) $\mrm{EXH}(\Phi)=\mrm{FIN}(\Phi)$ including the following: (ii) $\mrm{EXH}(\Phi)$ is an $F_\sigma$ subset of $\mathbb{R}^\om$;
(iii) $\mrm{EXH}(\Phi)$ does not contain a copy of $c_0$;  (iv) $\mrm{FIN}(\Phi)$ is separable; (v) all ideals representable in $\mrm{EXH}(\Phi)$ are $F_\sigma$. 
Most of the results in this section
turned out to be just reformulations of known, often even classical facts in Banach space theory. However, to the best of our knowledge the language of exhaustive and finite spaces was not present in the literature. This language seems to be quite convenient and reveals certain symmetries between the theory of Banach space and the theory of analytic P-ideals.

There is a natural way of associating an lsc submeasure $\varphi$ to an extended norm $\Phi$: if $\tau\in \mbb{R}^\om$, then $\varphi(A) = \Phi(P_A(\tau))$ is an lsc submeasure. So, an extended norm induces the Banach spaces $\mrm{EXH}(\Phi)$ and $\mrm{FIN}(\Phi)$, also it generates (many) analytic P-ideals. It seems interesting to study how the properties of Banach spaces induced by $\Phi$ impact properties of the ideals induced by $\Phi$ and vice versa. This is partially done
in Section \ref{main}. Also, concrete examples of Banach spaces motivate definitions of extended norms and it is interesting to see what concrete analytic P-ideals they induce. And
vice versa, one can consider extended norms associated to classical analytic P-ideals and look at the Banach spaces which they generate.

Pretentiously, one can summarize the above results saying that Banach spaces (with unconditional bases) are \emph{linearized} versions of (non-pathological) analytic P-ideals and, vice versa, (non-pathological) analytic P-ideals are \emph{discretized}
incarnations of Banach spaces (with unconditional
bases). We will show that this analogy leads us to new characterizations of classical properties, as well as, provides us with new tools of constructing interesting examples of ideals and Banach spaces.

From Section 5 on we consider Banach spaces and analytic P-ideals induced by families of finite sets. If $\mc{F}\subseteq[\om]^{<\om}=\{A\subseteq\om\colon |A|<\om\}$ covers $\om$ then $\Phi_\mc{F}(x)=\sup\{\sum_{i\in F}|x_i|\colon F\in\mc{F}\}$ is
an extended norm. The main result, Theorem \ref{compactly-supported}, of the paper points out a deeper interaction between combinatorial properties of $\mc{F}$, topological properties of $\mrm{EXH}(\Phi_\mc{F})$, and complexity of {\em
	$\mc{F}$-ideals}, that is, ideals of the form $\mrm{Exh}(\varphi)$ where $\varphi$ is induced by $\Phi$ in the way mentioned above(we will see that most important analytic P-ideals are of this form). We prove that the following are equivalent:
	(i) $\mc{F}$ is precompact, that is, $\overline{\mc{F}}\subseteq [\om]^{<\om}$ (see \cite{Jordi15} and \cite{Jordi-Stevo} for applications of precompact families), (ii) $\mrm{EXH}(\Phi_\mc{F})$ is $c_0$-saturated, (iii) $\mrm{EXH}(\Phi_\mc{F})$
	does not contain $\ell_1$, and (iv) nontrivial $\mc{F}$-ideals are not $F_\sigma$. 

We apply this result frequently in the next sections. In Section \ref{Ptak}, we present a combinatorial application related to Fremlin's DU problem; also we show that this application enables us to strengthen Pt\'{a}k's Lemma, a classical theorem in combinatorics and Mazur's Lemma, a basic tool in Banach space theory.

In Section \ref{schreier} we briefly discuss properties of the \emph{Schreier ideals} $\mathcal{I}_\alpha$ ($\alpha<\omega_1$),  ideals generated by the so called Schreier families $\mathcal{S}_\alpha$. We show that the following holds: (a) $\mathcal{I}_1$ is the ideal of asymptotic density zero sets; (b) $\mathcal{I}_\alpha$ is summable-like and not $F_\sigma$ for $\al>1$ (a quite
rare combination of properties); and (c) that $(\mc{I}_\al)$ forms a strictly $\subseteq$-decreasing sequence, intertwined by a sequence of $F_\sigma$ ideals.

In Section \ref{Schur}, we investigate Banach spaces induced by families which are far from being precompact. We show that the family associated with the so called Farah ideal induces a Banach space which possesses the Schur property but which is not isomorphic
to $\ell_1$. As a generalization of this example, we provide a sufficient condition for $\mrm{EXH}(\Phi_\mc{F})$ to satisfy the Schur property.

\subsection*{Acknowledgement} The authors would like to thank Jordi Lopez-Abad for his useful remarks and fruitful discussions on the subject of this article.

\section{Basics of analytic P-ideals}\label{analPsec}

We start with basic definitions, examples, and facts concerning analytic P-ideals. For more we refer the reader e.g. to \cite{hrusaksummary}. Recall that $\mc{I}\subseteq\mc{P}(\Omega)$ is an ideal on an infinite set $\Omega$ if it contains all
finite subsets of $\Omega$, it is hereditary (that is, $A\subseteq B\in\mc{I}$ implies $A\in\mc{I}$), and it is closed under taking unions of finitely many elements from it. In order to simplify the list of conditions in certain results, we allow
$\Omega\in\mc{I}$, i.e. $\mc{P}(\Omega)$ is considered to be an ideal; all other ideals are called proper. In our investigations $\Omega$ will always be countably infinite, and hence, without loss of generality, mostly we will work with $\Omega=\om$.
We will use the notations $\mc{I}^+=\mc{P}(\Omega)\setminus\mc{I}$ and $\mc{I}\clrest X=\{A\subseteq X\colon A\in\mc{I}\}$.

\medskip
An ideal $\mc{I}$ on $\om$ is {\em $F_\sigma$ } ({\em Borel, analytic}, etc) if $\mc{I}$ is an $F_\sigma$ (Borel, analytic, etc) subset of $\mc{P}(\om)$ identified with compact Polish space $2^\omega$ in the standard way. $\mc{I}$ is a {\em
P-ideal} if for each countable $\mc{C}\subseteq\mc{I}$ there is an
$A\in\mc{I}$ such that $C\subseteq^* A$ for each $C\in\mc{C}$ (where $C\subseteq^* A$ iff $C\setminus A$ is finite). Finally, $\mc{I}$ is {\em tall} if each infinite subset of $\om$ contains an infinite element of $\mc{I}$.

For example, the ideal of finite sets $\mrm{Fin}=[\om]^{<\om}$ is an $F_\sigma$ P-ideal which is not tall; the classical {\em summable ideal}
\[ \mc{I}_{1/n}=\bigg\{A\subseteq\om\setminus\{0\}\colon \sum_{n\in A}\frac{1}{n}<\infty\bigg\} \] is an $F_\sigma$ tall P-ideal; and the {\em
density zero ideal}
\[ \mc{Z}=\bigg\{A\subseteq\om\setminus\{0\}\colon  \frac{|A\cap [1,n]|}{n}\to 0\bigg\}\] is an $F_{\sigma\delta}$ (not $F_\sigma$)
tall P-ideal. We will see more general examples later.

\medskip
A function $\varphi\colon \mc{P}(\om)\to [0,\infty]$ is a {\em submeasure} (on $\om$) if $\varphi(\emptyset)=0$, ${\varphi}(X)\le \varphi(X\cup Y)\le \varphi(X)+\varphi(Y)$ whenever $X,  Y\subseteq\om$, and $\varphi(\{n\})<\infty$ for $n\in\om$. A submeasure $\varphi$ is
{\em lower semicontinuous} (lsc) if $\varphi(X)=\sup\{\varphi(F)\colon F\in [X]^{<\om}\}(=\lim_{n\to\infty}\varphi(X\cap n))$ for each $X\subseteq\om$. Submeasures turn out to be the ultimate tool in studying analytic P-ideals. If $\varphi$ is an lsc
submeasure on $\om$ then for $X\subseteq\om$ define $\mrm{tail}_\varphi(X)=\inf\{\varphi(X\setminus F)\colon F\in [\om]^{<\om}\}(=\lim_{n\to\infty}\varphi(X\setminus n))$,\footnote{The original notation for $\mrm{tail}_\varphi(X)$ was $\|X\|_\varphi$, we decided to use a new notation to avoid any confusions when working with Banach spaces.} and let
\begin{align*}
\mrm{Fin}(\varphi)= & \big\{A\subseteq\om\colon  \varphi(X)<\infty\big\},\\
\mathrm{Exh}(\varphi) = & \big\{X\subseteq\omega\colon \mrm{tail}_\varphi(X)=0\big\}.
\end{align*}
It is easy to see that $\mrm{Fin}(\varphi)$ is an $F_\sigma$ ideal,  and similarly $\mathrm{Exh}(\varphi)$ is an $F_{\sigma\delta}$ P-ideal. Clearly, $\mrm{Exh}(\varphi)\subseteq\mrm{Fin}(\varphi)$ always holds, and $\mrm{Exh}(\varphi)$ is tall iff $\varphi(\{n\})\to 0$.

\begin{thm}{\em (\cite{Mazur}, \cite{Solecki96})}\label{char} Let $\mc{I}$ be an ideal on $\om$. Then the following holds: (1) $\mc{I}$ is $F_\sigma$ iff $\mc{I}=\mrm{Fin}(\varphi)$ for some lsc $\varphi$; (2) $\mc{I}$ is an analytic P-ideal iff $\mc{I}=\mrm{Exh}(\varphi)$ for some lsc $\varphi$; and (3) $\mc{I}$ is an $F_\sigma$ P-ideal iff $\mc{I}=\mrm{Fin}(\varphi)=\mrm{Exh}(\varphi)$ for some lsc $\varphi$.
\end{thm}

In particular, analytic P-ideals are $F_{\sigma\delta}$. Let us overview some basic classes of analytic P-ideals.

\begin{exa}\label{exa-farah} {\em Summable ideals and Farah ideals.} Let $h\colon \om\to [0,\infty)$ be a function. 
The {\em summable ideal generated by $h$} is
\[ \mc{I}_h=\bigg\{A\subseteq\om\colon \sum_{n\in A} h(n)<\infty\bigg\}.\]
Summable ideals are $F_\sigma$ P-ideals, and $\mc{I}_h$ is tall iff $h(n)\to 0$. Of course, $\mc{I}_{1/n}$ is a summable ideal up to an bijection between $\om$ and $\om\setminus \{0\}$ (similarly, properties/classes of ideals on $\om$ are extended to ideals on arbitrary countably infinite sets).

In \cite[Example 1.11.1]{Farah} Farah gave an example of an  $F_\sigma$ P-ideal which is not a summable ideal. Motivated by his example we define a class of ideals which we call Farah ideals and which contain all summable ideals: Let $(P_n)$ be a partition of $\om$ into finite nonempty sets, and let $\vec\vartheta=(\vartheta_n)$ be a sequence of submeasures, $\vartheta_n\colon \mc{P}(P_n)\to [0,\infty)$. 
The {\em Farah ideal generated by $\vec\vartheta$} is
\[ \mc{J}_{\vec\vartheta}=\bigg\{A\subseteq\om\colon \sum_{n\in\om}\vartheta_n(A)<\infty\bigg\},\]
where $\vartheta_n(A)=\vartheta_n(A\cap P_n)$. All these ideals are $F_\sigma$ P-ideals, and $\mc{J}_{\vec\vartheta}$ is tall iff $\max\{\vartheta_n(\{i\})\colon i\in P_n\}\xrightarrow{n\to\infty}0$. Clearly, if all $\vartheta_n$ are measures then $\mc{J}_{\vec\vartheta}$ is a summable ideal and of course all summable ideals are obviously Farah ideals.
Notice that $\varphi(A)=\sum_{n\in\om}\vartheta_n(A)$ is an lsc submeasure and $\mc{J}_{\vec\vartheta}=\mrm{Fin}(\varphi)=\mrm{Exh}(\varphi)$.

Farah's original example on $\om\setminus\{0\}$ is the special case when $P_n=[2^n,2^{n+1})$, $\vartheta_0(\{1\})=1$, and $\vartheta_n=\min\{n,|A|\}/n^2$ for $A\subseteq P_n$ if $n>0$.
\end{exa}

\begin{exa}{\em Density and generalized density ideals.} Let $(P_n)$ and $\vec{\vartheta}=(\vartheta_n)$ be as above. 
The {\em generalized density ideal generated by $\vec{\vartheta}$}
is
\[ \mc{Z}_{\vec{\vartheta}}=\bigg\{A\subseteq\om\colon \limsup_{n\to\infty}\vartheta_n(A)=0\bigg\}.\]
These ideals are $F_{\sigma\delta}$ P-ideals, and $\mc{Z}_{\vec{\vartheta}}$ is tall iff $\max\{\vartheta_n(\{i\})\colon i\in P_n\}\xrightarrow{n\to\infty}0$. 
Clearly $\mc{J}_{\vec\vartheta}\subseteq\mc{Z}_{\vec\vartheta}$. An ideal $\mc{I}$ is a {\em density ideal} if $\mc{I} = \mc{Z}_{\vec{\vartheta}}$ where all submeasures in the sequence are measures.
Clearly, $\psi(A)=\sup\{\vartheta_n(A)\colon n\in\om\}$ is an lsc submeasure and $\mc{Z}_{\vec\vartheta}=\mrm{Exh}(\psi)$.

For example, as
$\mc{Z}=\{A\subseteq\om\setminus\{0\}\colon |A\cap [2^n,2^{n+1})|/2^n\to 0\}$, the density zero ideal is a density ideal.
\end{exa}

Clearly, $\mrm{Fin}$, and in general, trivial modifications of $\mrm{Fin}$, that is, ideals of the form $\{A\subseteq\om\colon |A\cap X|<\om\}$ for some $X\subseteq\om$ are both Farah ideals and a (generalized) density ideals. We will show that apart from these examples, these classes are disjoint.

\begin{thm}\label{nonfsigmachar} (reformulation of \cite[Theorem 3.3]{Solecki96}) An analytic P-ideal $\mc{I}$ is not $F_\sigma$ iff there is a decreasing sequence $(X_n)$ in $\mc{I}^+$ which has no $\mc{I}$-positive pseudointersection (that is, if $X\subseteq^* X_n$ for every $n$, then $X\in\mc{I}$).
\end{thm}
\begin{fact}\label{nowheretall} (see \cite[Fact 2.3]{Ways19})
Let $\mc{I}$ be an analytic P-ideal. Then the following are equivalent: (i) $\mc{I}$ is nowhere tall, that is, $\mc{I}\clrest X$ is not tall for any $X\in\mc{I}^+$. (ii) If $\mc{I}=\mrm{Exh}(\varphi)$ then $\mc{I}=\{A\subseteq\om\colon A$ is finite
or $\lim_{n\in A}\varphi(\{n\})=0\}$. (iii) $\mc{I}$ is a trivial modification of $\mrm{Fin}$, or $\mc{I}$ is isomorphic (via a bijection between $\om$ and $\om\times\om$) to the density ideal $\{\0\}\otimes\mrm{Fin}=\{A\subseteq\om\times\om\colon \forall n  \ \{k\colon (n,k)\in A\}$ is finite$\}$.
\end{fact}
\begin{cor}\label{nowheretallcor}
A generalized density ideal $\mc{Z}_{\vec\vartheta}$ is either a trivial modifications of $\mrm{Fin}$, or it is not $F_\sigma$ (in particular, not a Farah ideal).
\end{cor}
\begin{proof}
Applying Theorem \ref{nonfsigmachar}, the sequence $X_n=(\om\setminus n)\times\om$ witnesses that $\{\0\}\otimes\mrm{Fin}$ is not $F_\sigma$. Now if $\mc{Z}_{\vec\vartheta}$ is neither this ideal nor a trivial modification of $\mrm{Fin}$, then,
according to Fact \ref{nowheretall},  there is an $X\in\mc{Z}_{\vec\vartheta}^+$ such that $\mc{Z}_{\vec\vartheta}\clrest X$ is tall. Notice that $\mc{Z}_{\vec\vartheta}\clrest X$ is also a generalized density ideal: If $\vartheta_n\colon
\mc{P}(P_n)\to[0,\infty)$, $Q_n=X\cap P_n$, and $\eta_n=\vartheta_n\clrest\mc{P}(Q_n)$, then $\mc{Z}_{\vec\vartheta}\clrest X=\mc{Z}_{\vec\eta}$. Let $\psi$ be the canonical lsc submeasure generating $\mc{Z}_{\vec\eta}$, that is,
	$\psi(A)=\sup\{\eta_n(A)\colon n\in\om\}$.

We show that if $Y\subseteq X$ is $\mc{Z}_{\vec\eta}$-positive and $0<\eps<\mrm{tail}_\psi(Y)$, then there is a $Z\subseteq Y$ such that $\mrm{tail}_\psi(Z)=\eps$. This is enough because then we can construct a sequence $Y_0\supseteq Y_1\supseteq\cdots$ in $\mc{Z}_{\vec\eta}^+$ such that $\mrm{tail}_\psi(Y_n)\to 0$, in particular, all pseudointersections of this sequence belong to $\mc{Z}_{\vec\eta}$. Applying Theorem \ref{nonfsigmachar}, $\mc{Z}_{\vec\eta}$ is not $F_\sigma$, and hence neither is $\mc{Z}_{\vec\vartheta}$.

If $Y$ and $\eps$ are as above, then there is an $E\in [\om]^\om=\{A\subseteq \om\colon |A|=\om\}$ such that $\eta_n(Y)>\eps$ for every $n\in E$. As $\mc{Z}_{\vec\eta}$ is tall, we know that $\max\{\eta_n(\{i\})\colon i\in Q_n\}\to 0$, and hence we
can assume that $\forall n\in E \ \forall i\in Q_n \ \eta_n(\{i\})<\eps$.
Now, for each $n$, fix a $Z_n\subseteq Y\cap Q_n$ such that $\eta_n(Z_n)\leq \eps$ but $\eta_n(Z_n\cup\{i\})>\eps$ for every $i\in (Y\cap Q_n)\setminus Z_n$. Then clearly $\eta_n(Z_n)\to\eps$ and hence $Z=\bigcup_{n\in E}Z_n$ is as required.
\end{proof}

There are many analytic P-ideals which are neither $F_\sigma$ nor generalized density ideals; one of the most
important example of such an ideal is the following:

\begin{exa}\label{exa-trN} {\em Trace ideal of the null ideal.} The {\em $G_\delta$-closure} of a set $A\subseteq 2^{<\om}=\{s\colon s$ is a function, $\dom(s)\in\om$, and $\ran(s)\subseteq 2\}$ is defined as $[A]_\delta=\{x\in 2^\om\colon \exists^\infty n
	\ x\clrest n\in A\}$ where $\exists^\infty$ stands for ``there exists infinitely many'' (and dually, $\forall^\infty$ means ``for all but finitely many''). Then $[A]_\delta$ is a $G_\delta$ subset of $2^\om$, and every $G_\delta$ subset of $2^\om$ is of this form. Let $\mc{N}$ denote the $\sigma$-ideal of null subsets of $2^\om$ with respect to the usual product probability measure. The trace of $\mc{N}$ is defined as follows:
\[ \mathrm{tr}(\mc{N})=\big\{A\subseteq 2^{<\om}\colon [A]_\delta\in \mc{N}\big\}.\]
This is a $F_{\sigma\delta}$ tall P-ideal. If $\varphi(A)=\sum\{2^{-|s|}\colon s\in A$ is $\subseteq$-minimal$\}$, then $\varphi$ is an lsc submeasure and $\mrm{tr}(\mc{N})=\mrm{Exh}(\varphi)$.

We know (see \cite{Hrusak-Hernandez}) that $\mrm{tr}(\mc{N})$ is {\em totally bounded}, that is, whenever $\mrm{tr}(\mc{N})=\mrm{Exh}(\psi)$ for some lsc $\psi$, then $\psi(2^{<\om})<\infty$, in particular, applying Theorem \ref{char} (1), $\mrm{tr}(\mc{N})$ is not $F_\sigma$. Also, we know (see \cite[Proposition 6.2]{Our-paper}) that $\mrm{tr}(\mc{N})$ {\em summable-like} (see the definition later), hence it is not a generalized density ideal either. 
\end{exa}

\section{Basics of Banach spaces}
In this section we recall some basic notions and results from the theory of Banach spaces, see any classical textbook on the subject for more details (e.g. Schlumprecht's online available lecture notes \cite{Schlumprecht}). All Banach spaces we work with are considered over $\mbb{R}$. A sequence $(e_n)$ in a Banach space $X$ is {\em basic} if for every $x\in [(e_n)]\colon =\overline{\mrm{span}}(\{e_n\colon n\in\om\})$ there is a unique $(\al^x_n)\in\mbb{R}^\om$ such that $x=\sum_{n\in\om}\al^x_ne_n$. A basic sequence $(e_n)$ in $X$ is a {\em (Schauder) basis} of $X$ if $X=[(e_n)]$.

Two basic sequences $(e_n)$ in $X$ and $(v_n)$ in $Y$ are {\em (isometrically) equivalent} if there is an (isometric) isomorphism $T\colon [(e_n)]\to [(v_n)]$ such that $T(e_n)=v_n$ for every $n$.

\medskip
Recall that if $(x_n)$ is sequence in $X$, then we say that $\sum_{n\in\om}x_n$ is {\em unconditionally convergent} if \[ \exists\; x\in X\; \forall\; \eps>0\;\exists\; N\in \om\; \forall\; F\in [\om]^\om\;\bigg(N\subseteq F\rightarrow \bigg\|x-\sum_{n\in F}x_n\bigg\|<\eps\bigg).\]
Let us mention a couple of equivalent statements: (a) $\forall$ $\eps>0$ $\exists$ $N\in\om$ $\forall$ $E\in [\om\setminus N]^{<\om}$ $\|\sum_{n\in E}x_n\|<\eps$; (b) $\forall$ permutation $\pi\colon \om\to\om$ $\sum_{n\in\om}x_{\pi(n)}$ is convergent; (c) $\forall$ $A\subseteq\om$ $\sum_{n\in A}x_n$ is convergent; and (d) $\forall$ $(\eps_n)\in \{\pm 1\}^\om$ $\sum_{n\in\om}\eps_nx_n$ is convergent.

The classical Dvoretzky-Rogers theorem (see \cite{Dvoretzky-Rogers}) says that in each infinite dimensional $X$ there is a sequence $(x_n)$ such that $\sum_{n\in\om} x_n$ is unconditionally convergent but $\sum_{n\in\om}\|x_n\|=\infty$.

\medskip
Now, a basic sequence (basis, resp.) $(e_n)$ is {\em unconditional} if $x=\sum_{n\in\om}\al^x_ne_n$ converges unconditionally for every $x\in [(e_n)]$. For every unconditional basic sequence (basis, resp.) $(e_n)$ there are $C,K\in [1,\infty)$ such that for every $n\in\om$, $F\subseteq n$, $\al_k\in\mbb{R}$, and $\eps_k=\pm1$ ($k<n$) the following holds:
\[\bigg\|\sum_{k\in F}\al_ke_k\bigg\|\leq C\bigg\|\sum_{k<n}\al_ke_k\bigg\|\;\;\text{and}\;\;
\bigg\|\sum_{k<n}\eps_k\al_ke_k\bigg\|\leq K\bigg\|\sum_{k<n}\al_ke_k\bigg\|.\]
Moreover, the following stronger form of the second inequality also holds: \[ \bigg\|\sum_{k<n}\al_ke_k\bigg\|\leq K\bigg\|\sum_{k<n}\be_ke_k\bigg\|\;\,\text{whenever $|\al_k|\leq|\be_k|$.}\]
If $K$ is such a constant, then we say that the base $(e_n)$ is {\em $K$-unconditional}. If $C$ and $K$ are the smallest such that the above inequalities hold, then $C\leq K\leq 2C$. Also, if any of these inequalities holds for a sequence $(e_n)$ of non zero elements from $X$ (with all possible parameters in the inequality and fixed $C$ or $K$), then $(e_n)$ is a ($K$-)unconditional basic sequence in $X$.

If $(e_n)$ is a(n unconditional) basis in $X$ then the associated {\em coordinate functionals} $e^*_n\colon X\to\mbb{R}$ are determined by $x=\sum_{n\in\om}e^*_n(x)e_n$ for every $x$. Then $e_n^*\in X^*=\{$bounded linear functionals on $X\}$ and $(e^*_n)$ is a(n unconditional) basic sequence in $X^*$.

\medskip
We will need two special types of bases:

We say that a basis $(e_n)$ in $X$ is {\em shrinking} if $(e^*_n)$ is a basis in $X^*$ (i.e. $[(e^*_n)]=X^*$); equivalently, $\forall x^*\in X^* \ \|x^*\clrest [(e_n)_{n\geq N}]\|\xrightarrow{N\to\infty}0$; equivalently, $(e_n^*)$ is a boundedly complete basis (see below) in $[(e_n^*)]$. Furthermore, if $(e_n)$ is an unconditional basis in $X$, then $(e_n)$ is shrinking iff $X$ does not contain a copy of $\ell_1$. For example, the standard basis of $c_0$ is shrinking.

We say that a basis $(e_n)$ in $X$ is {\em boundedly complete} if $\sum_{n\in\om}a_ne_n$ is convergent whenever $\sup\{\|\sum_{n<N}a_ne_n\|\colon N\in\om\}<\infty$; equivalently, the natural embedding $X\to [(e_n^*)]^*$ is an isomorphism; equivalently, $(e_n^*)$ is a shrinking basis in $[(e_n^*)]$.  Furthermore, if $(e_n)$ is an unconditional basis in $X$, then $(e_n)$ is boundedly complete iff $X$ does not contain a copy of $c_0$. For example, the standard basis of $\ell_1$ is boundedly complete.

\medskip
In general, characterisations of those spaces which contain copies of $c_0$ or $\ell_1$ are among the most well-studied fundamental problems in the theory of Banach spaces. Let us here recall two related results we will need later. We say that a
sequence $(x_n)$ in $X$ is \emph{perfectly bounded} if $\sup\{ \|\sum_{n\in F}x_n \|\colon F\in [\om]^{<\om}\}<\infty$.

\begin{thm}\label{Bessaga} {\em (Bessaga-Pe\l czy\'nski $c_0$ theorem, see e.g. \cite{Bessaga58})}
Let $X$ be a Banach space. Then $X$ does not contain a copy of $c_0$ iff $\sum_{n\in \omega} x_n$ is unconditionally convergent for every perfectly bounded sequence $(x_n)$ from $X$.
\end{thm}

In the next theorem we need certain notions associated to the weak topologies of Banach spaces, we give rather direct definitions here without defining the weak topology itself: A sequence $(a_n)$ in $X$ {\em converges weakly} to $a\in X$ if $x^*(a_n)\to x^*(a)$ for every $x^*\in X^*$; and similarly, $(a_n)$ is {\em weakly Cauchy} if $(x^*(a_n))$ is Cauchy/convergent in $\mbb{R}$ for every $x^*\in X^*$.

\begin{thm}\label{Ros} {\em (Rosenthal $\ell_1$ theorem, see \cite{Rosenthal-l1}}). Let $X$ be a Banach space and $(a_n)$ a bounded sequence in $X$. Then $(a_n)$ has a subsequence $(a_n)_{n\in E}$ ($E\in [\om]^\om$) such that (exactly) one of the following holds: (a) $(a_n)_{n\in E}$ is weakly Cauchy; (b) $(a_n)_{n\in E}$ is equivalent to the standard basis of $\ell_1$.
\end{thm}

At the end of this section, as the first application of Banach space theory to analytic P-ideals, we will present a natural common generalization of Farah ideals and generalized density ideals (motivated by examples from \cite{LV}).

\begin{exa} \label{farah-and-generalized}{\em Common generalization of Farah and generalized density ideals.} Let $X\subseteq\mbb{R}^\om$ be a linear space equipped with a norm $\|\cdot\|$ generating a complete topology such that the canonical algebraic basis
$(e_n)$ of $c_{00}=\{x=(x_n)\in\mbb{R}^\om\colon \forall^\infty$ $n$ $x_n=0\}$ is a $1$-unconditional base of $X$ (e.g. $X=c_0$ or $X=\ell_p$, $1\leq p<\infty$). Let $(P_n)$ and $\vec\vartheta=(\vartheta_n)$ be as in the definitions of
$\mc{J}_{\vec\vartheta}$ and $\mc{Z}_{\vec\vartheta}$. 
We associate the following family to $\vec\vartheta$ and $X$:
\[ \mc{J}_{\vec\vartheta,X}= \big\{A\subseteq\om\colon \big(\vartheta_0(A),\vartheta_1(A),\dots\big)\in X\big\}\]
where of course $\vartheta_n(A)=\vartheta_n(A\cap P_n)$. Applying that $(e_n)$ is unconditional, it is easy to show that this family is an ideal. Now if $\vec\vartheta$ and $X$ are as above, then define $\varphi_{\vec\vartheta,X}\colon \mc{P}(\om)\to[0,\infty]$ as follows:
\[\varphi_{\vec\vartheta,X}(A)=\sup\bigg\{\bigg\|\sum_{n\in\om}\vartheta_n(F)e_n\bigg\|\colon F\in [A]^{<\om}\bigg\}.\]
Applying $1$-unconditonality of $(e_n)$, it is easy to show that $\varphi_{\vec\vartheta,X}$ is subadditive, and of course, it is monotone and lsc by definition; furthermore, a straightforward argument shows that $\mc{J}_{\vec\vartheta,X}=\mrm{Exh}(\varphi_{\vec\vartheta,X})$, hence $\mc{J}_{\vec\vartheta,X}$ is an analytic P-ideal.

Notice that $\mc{J}_{\vec\vartheta}=\mc{J}_{\vec\vartheta,\ell_1}$ and $\mc{Z}_{\vec\vartheta}=\mc{J}_{\vec\vartheta,c_0}$.
\end{exa}

\begin{rem}
An easy application of the Bessaga-Pe\l cz\'nski $c_0$ theorem shows that the ideal $\mc{J}_{\vec\vartheta,X}$ is $F_\sigma$ whenever $X$ does not contain a copy of $c_0$ .
\end{rem}

\section{Representing ideals in Banach spaces}\label{reprsetting}

In \cite{Our-paper} the authors considered a natural generalization of summable ideals. Let $G$ be a completely metrizable Abelian group. Unconditional convergence of infinite sums in $G$ is defined just like in Banach spaces, and the appropriate translations of its equivalent formalizations hold as well. Now, fix a sequence $h\colon \omega \to G$ and define
\[ \mathcal{I}^G_h= \bigg\{A\subseteq\omega\colon \sum_{n\in A} h(n)\;\text{is unconditionally convergent in}\;G\bigg\}. \]
Then $\mc{I}^G_h$ is an ideal on $\omega$, and ideals of this form are called  \emph{representable} in $G$.

\begin{exa} (see \cite{Our-paper})
Summable ideals are exactly those represented in $\mathbb{R}$. Similarly, an ideal is representable in $\mbb{R}^\om$ iff it is the intersection of countable many summable ideals (such an ideal is not necessarily $F_\sigma$). Finally, we show that $\mc{Z}$ is representable in $c_0$ (over $\om\setminus \{0\}$ for now):  $\mc{Z}=\mc{I}^{c_0}_h$ where $h(n)=2^{-k}e_k$ if $n\in [2^k,2^{k+1})$ and $(e_k)$ is the standard basis of $c_0$ (similarly, every non-pathological generalized density ideal is representable in $c_0$).
\end{exa}

One of the main theorems of \cite{Our-paper} says that the class of ideals represented in Polish Abelian groups coincide with the class of analytic P-ideals, and the class of those represented in Banach spaces is still quite large containing all important examples. Allow us to first state the theorem, and give a detailed introduction to the notion we use in it afterwards:

\begin{thm}\label{PAGBSrepr} (\cite{Our-paper}) An ideal is representable in a Polish Abelian group iff it is an analytic P-ideal. An ideal is representable in a Banach space iff it is a non-pathological analytic P-ideal.
\end{thm}

If $\wt{\mu}=\{\mu_\al\colon \al<\ka\}$ is a family of measures on $\om$ such that $\sup\{\mu_\al(\{n\})\colon \al<\ka\}<\infty$ for every $n$, then their pointwise supremum $\varphi_{\wt{\mu}}(A)=\sup\{\mu_\al(A)\colon \al<\ka\}$ is always an lsc
submeasure. Submeasures of this form are called {\em non-pathological}. Notice that if $\varphi$ is non-pathological, then  there is a countable family $\wt{\mu}=\{\mu_n\colon n\in\om\}$ of measures such that $\varphi(F)=\varphi_{\wt{\mu}}(F)$ for
every $F\in [\om]^{<\om}$, and hence (by lsc) $\varphi=\varphi_{\wt{\mu}}$. Furthermore, if we replace $\mu_n$ with $\{\mu_n\clrest E\colon E\in [\om]^{<\om}\}$, then $\varphi_{\wt{\mu}}$ does not change, and hence we can assume that
$\mrm{supp}(\mu_n)=\{k\in\om\colon \mu_n(\{k\})\ne\0\}$ is finite for every $n$. Actually, a trivial compactness argument shows that if $\varphi$ is non-pathological, then for every $F\in [\om]^{<\om}$ there is a measure $\mu_F$ such that
$\mrm{supp}(\mu_F)\subseteq F$, $\mu_F\leq \varphi$, and $\mu_F(F)=\varphi(F)$; in particular $\varphi=\sup\{\mu_F\colon F\in [\om]^{<\om}\}$. Moreover, applying the Banach-Alaoglu theorem, one can show that this hold for every $F\in\mrm{Fin}(\varphi)$.

Also, by extending $\wt{\mu}$ with countable many measures of finite support, we can always assume (without changing $\mrm{Exh}(\varphi_{\wt{\mu}})$) that $\bigcup_{n\in\om}\mrm{supp}(\mu_n)=\om$, i.e. that $\varphi_{\wt{\mu}}(\{n\})>0$ for every
$n$. From now on when using a family $\wt{\mu}$ of measures, we will always assume that $\wt{\mu}=\{\mu_n\colon n\in\om\}$ is countable, $\mrm{supp}(\mu_n)$ is finite for every $n$, and  $0<\sup\{\mu_n(\{k\})\colon n\in\om\}<\infty$ for every $k$.

\medskip
We say that an analytic P-ideal $\mc{I}$ is {\em non-pathological} if there is a non-pathological $\varphi$ such that  $\mc{I}=\mrm{Exh}(\varphi)$ (it is trivial to construct pathological submeasures, for examples of pathological ideals see
\cite[Theorem 1.9.5]{Farah}). For example, if $\vartheta_n$ is non-pathological for every $n$, then the Farah ideal $\mc{J}_{\vec\vartheta}$ and the generalized density ideal $\mc{Z}_{\vec\vartheta}$ are non-pathological. Trace of null is also
non-pathological: Let $\mc{F}=\{F\subseteq 2^{<\om}\colon F$ is a finite antichain$\}$, for every $F\in\mc{F}$ define the measure $\mu_F$ as $\mu_F(A)=\sum_{s\in A\cap F}2^{-|s|}$, and let $\wt{\mu}=\{\mu_F\colon F\in\mc{F}\}$. Then $\mrm{tr}(\mc{N})=\mrm{Exh}(\varphi_{\wt{\mu}})$.

\medskip
\begin{proof}[Remarks on the proof of Theorem \ref{PAGBSrepr}] Let $G$ and $h$ be as above and fix an invariant compatible (hence complete) metric $d$ on $G$. Then $\varphi^G_h(A)=\sup\{d(0,\sum_{n\in F}h(n))\colon F\in [A]^{<\om}\}$ is an lsc submeasure and $\mc{I}^G_h=\mrm{Exh}(\varphi^G_h)$.

\smallskip
If $\mc{I}$ is represented in a Banach space $X$, then it is is represented in $\ell_\infty$, $\mc{I}=\mc{I}^{\ell_\infty}_h$,  $h(n)=(h(n)_0,h(n)_1,\dots)$. Switching to  $h'(n)=(|h(n)_0|,|h(n)_1|,\dots)$ does not change the ideal, and if we define
the measures $\mu_k(A)=\sum_{n\in A}|h(n)_k|$ then $\varphi^{\ell_\infty}_{h'}=\sup\{\mu_k\colon k\in\om\}$ is non-pathological. If we do not wish to embed $X$ in $\ell_\infty$, then the above construction pulled back to $X$ gives us the
non-pathological submeasure \[ \psi^X_h(A)=\sup\bigg\{\bigg\|\sum_{n\in F}\eps_nh(n)\bigg\|\colon F\in [A]^{<\om}\;\text{and}\;(\eps_n)\in \{\pm 1\}^F\bigg\}\]
and $\mc{I}^X_h=\mrm{Exh}(\psi^X_h)$.
\end{proof}

There is an important special class of non-pathological submeasures: for every cover $\mc{F}\subseteq [\om]^{<\om}$ of $\om$ and sequence $\tau=(\tau_n)\in(0,\infty)^\om$, define $\wt{\mu}(\mc{F},\tau)$ as follows: $\mu\in \wt{\mu}(\mc{F},\tau)$ iff
$\mu$ is a measure on $\om$, $\mrm{supp}(\mu)\in \mc{F}$, and $\mu(\{n\})=\tau_n$ for $n\in\mrm{supp}(\mu)$; then $\wt{\mu}(\mc{F},\tau)$ is a family of measures as above. Let $\varphi_{\mc{F},\tau}=\varphi_{\wt{\mu}(\mc{F},\tau)}$ and
$\mc{I}_{\mc{F},\tau}=\mrm{Exh}(\varphi_{\mc{F},\tau})$. Most of our main examples of analytic P-ideals can be easily written of this form: in the case of $\mc{I}_h$ let $\mc{F}=[\om]^{<\om}$ and $\tau_n=h(n)$. In the case of Farah's example let
$\mc{F}=\{F\in [\om\setminus\{0\}]^{<\om}\colon \forall$ $n>0$ $|F\cap [2^n,2^{n+1})|= n\}$, $\tau_1=1$, and $\tau_k=1/n^2$ for $k\in [2^n,2^{n+1})$ and $n>0$.  Finally, it is easy to present density ideals and $\mrm{tr}(\mc{N})$ in this form.

As we will see later, combinatorics of ideals of this simple form is strongly related to properties of certain easily definable but interesting Banach spaces.

\begin{prob}
Can all non-pathological analytic P-ideals be written of the form $\mc{I}_{\mc{F},\tau}$?
\end{prob}

We close this section with some general results on representability of ideals in Banach spaces. In \cite{Our-paper} the authors considered questions of representability of ideals in classical Banach spaces. The case of $c_0$ seems to be particularly interesting. As summable ideals can be represented in $\mathbb{R}$, they are representable in every Banach space. The next theorem says that apart from these trivial representations no $F_\sigma$ P-ideals can be represented in $c_0$, whereas, as we saw above, $c_0$ is a ``natural''
Banach space for representing density ideals.

\begin{thm}\label{our-theorem}
Among $F_\sigma$ tall P-ideals only summable ideals are representable in $c_0$.
\end{thm}

Representability of ideals was considered independently by Drewnowski and Labuda (in a bit different setting). They proved another theorem revealing some connections between $c_0$ and representability of $F_\sigma$ ideals.

\begin{thm}\label{DrLa} {\em (see \cite{Drewnowski10} and \cite{Drewnowski17})} If $X$ is an infinite dimensional Banach space, then there is an ideal representable in $X$ which is not a summable ideal; furthermore $X$ does not contain a copy of $c_0$ iff $\mc{I}^X_h$ is $F_\sigma$ for every $h\colon \om\to X$.
\end{thm}
\begin{proof}
To prove the first statement, fix a sequence $(x_n)$ such that $\sum_{n\in\om} x_n$ is unconditionally convergent but $\sum_{n\in\om}\|x_n\|=\infty$. We can assume that $0<\|x_n\|<1$ for every $n$. Fix a partition $(P_n)$ of $\om$ into finite sets such that $1\leq |P_n|\cdot\|x_n\|<2$ for every $n$ and define $h\colon \omega \to X$ as $h(k)=x_n$ for every $k\in
	P_n$. We claim that that $\mc{I}^X_h$ is not a summable ideal. Assume on the contrary that $\mc{I}^X_h=\mc{I}_f$ for some $f\colon \omega\to [0,\infty)$. Then $\liminf_{n\in\om}\sum_{k\in P_n}f(k)=\eps>0$ because $\|\sum_{k\in P_n}h(k)\|=|P_n|\cdot\|x_n\|\geq 1$ and hence $\bigcup_{n\in E}P_n\in\mc{I}^+_f$ for every infinite $E\subseteq\om$. Now if $k_n\in P_n$ such that $f(k_n)=\max\{f(k)\colon k\in P_n\}$, then on the one hand, as $\sum_{n\in\om} h(k_n)=\sum_{n\in\om}x_n$ is unconditionally convergent, $\sum_{n\in\om}f(k_n)<\infty$ follows, but on the other hand $f(k_n)\geq \eps/|P_n|> \eps\|x_n\|/2$ and hence $\sum_{n\in\om}f(k_n)\geq \sum_{n\in\om}\eps\|x_n\|/2=\infty$, a contradiction.

\smallskip
Second part of the statement: $\mc{Z}$ is representable in $c_0$ and we know that $\mc{Z}$ is not $F_\sigma$.
Conversely, recall that $\varphi(A)=\varphi^X_h(A)=\sup\{\|\sum_{k\in F} h(k)\|\colon F\in
	[A]^{<\om}\}$ is an lsc submeasure and $\mc{I}^X_h=\mrm{Exh}(\varphi)$ (see the remarks on the proof of Theorem \ref{PAGBSrepr}). Now, by Theorem \ref{Bessaga}, if $X$ does not contain a copy of $c_0$, then $A\in\mc{I}^X_h$ iff $\sum_{n\in A} h(n)$ is unconditionally convergent iff $h\clrest A$ is perfectly bounded iff $\varphi(A)<\infty$. Therefore, $\mc{I}^X_h = \mrm{Fin}(\varphi)$ is $F_\sigma$.
\end{proof}

\section{The Banach spaces $\mrm{FIN}(\Phi)$ and $\mrm{EXH}(\Phi)$}\label{main}

The notion of representability of ideals in Banach spaces connects the theories of Banach space and of analytic P-ideals. One can consider questions of the form: Which ideals are representable in a Banach spaces satisfying certain properties or vice versa, what do we know about Banach spaces in which an ideal or family of ideals is represented?

On the other hand, an ideal can be represented in a Banach space in many different ways and sometimes the fact that an ideal can
be represented in a Banach space does not provide much information if we do not assume something more about the way it is represented. The extreme example illustrating this remark is the case of summable ideals: they are
represented in every Banach space.

In this section we will present a more canonical way of representing ideals in Banach spaces which enables us to avoid this issue. Through a general construction of Banach spaces, analogous to the way we associate ideals to
submeasures, we show that all non-pathological analytic P-ideals have representations via $1$-unconditional basic sequences. As a byproduct of this approach we obtain some structure theorems for Banach spaces (most of them well-known, though).

\medskip
An {\em extended norm} is a function $\Phi\colon \mbb{R}^\om \to [0,\infty]$ satisfying all properties of a norm but possibly taking infinite values (where we let $0\cdot\infty= 0$ at the axiom $\Phi(\al x)=|\al|\Phi(x)$). To avoid strange examples of extended norms, more precisely, to deepen the analogy with lsc submeasures, we will assume that extended norms satisfy some natural additional conditions, we say that $\Phi$ is {\em nice} if the following holds:
 \begin{itemize}
\item[(a)] $\forall x\in c_{00}=\{(x_n)\in\mbb{R}^\om\colon \forall^\infty$ $n$ $x_n=0\} \  \Phi(x)<\infty$;
\item[(b)] (monotonicity) If $x,y\in\mbb{R}^\om$ and $|x_n|\leq |y_n|$ for every $n$, then  $\Phi(x) \leq \Phi(y)$;
\item[(c)] (lower semicontinuity, lsc) $\forall x\in\mbb{R}^\om  \ \Phi(P_n(x))\xrightarrow{n\to \infty}\Phi(x)$
\end{itemize}
where $P_A\colon \mbb{R}^\om\to\mbb{R}^\om$ is the projection associated to $A\subseteq
\omega$, that is, $P_A(x)(n)=x(n)$ if $n\in A$ and $0$ otherwise. Now
fix a nice extended norm $\Phi$, define  $\mrm{tail}_\Phi(x)=\inf\{\Phi(P_{\om\setminus F}(x):F\in [\om]^{<\om})(=\lim_{n\to\infty}\Phi(P_{\om\setminus n}(x)))$, and consider the following linear spaces:
\begin{align*}
\mathrm{FIN}(\Phi) &= \big\{x\in \mbb{R}^\om\colon \Phi(x)<\infty\big\},\\
\mrm{EXH}(\Phi) &= \big\{x\in \mbb{R}^\om\colon\mrm{tail}_\Phi(x) = 0\big\}.
\end{align*}

When working with these spaces, $(e_n)$ denotes the standard (algebraic) basis of $c_{00}$.

\begin{prop} Let $\Phi$ be a nice extended norm. Then $\mrm{EXH}(\Phi)\subseteq \mrm{FIN}(\Phi)$ equipped with $\Phi$ are Banach spaces, $\mrm{EXH}(\Phi)$ is the completion of $c_{00}$,  and $(e_n)$ is a $1$-unconditional basis in $\mrm{EXH}(\Phi)$. Furthermore, every $1$-unconditional basis $(b_n)$ in a Banach space $X$ is isometrically equivalent to $(e_n)$ in $\mrm{EXH}(\Phi)$ for some nice $\Phi$. In particular, every unconditional basis is equivalent to $(e_n)$ in
	$\mrm{EXH}(\Phi)$ for some nice $\Phi$.
\end{prop}
\begin{proof}
As $\Phi$ is finite on $c_{00}$ and $\Phi(x)\leq\Phi(P_n(x))+\Phi(P_{\om\setminus n}(x))$, $\mrm{EXH}(\Phi)\subseteq\mrm{FIN}(\Phi)$ follows.

\smallskip
$\mrm{FIN}(\Phi)$ is complete: Let $(x_n)$ be a Cauchy sequence in $\mrm{FIN}(\Phi)$. Applying monotonicity,  $\Phi(P_{\{k\}}(x_n-x_m))\leq \Phi(x_n-x_m)$ for every $k,n,m$, and hence $(P_{\{k\}}(x_n))_{k\in\om}$ is Cauchy in the $k$th $1$-dimensional
coordinate space of $\mbb{R}^\om$ (a Banach space, as $\Phi$ is finite on $c_{00}$), $P_{\{k\}}(x_n)\xrightarrow{n\to\infty}y_k$, let $y=(y_k)$. We will first show that $y\in\mrm{FIN}(\Phi)$. The sequence $\{x_n\colon n\in\om\}$ is bounded, let
say $\Phi(x_n)\leq B$ for every $n$. We show that $\Phi(y)\leq 2B$, i.e. (by lsc of $\Phi$) $\Phi(P_M(y))\leq 2B$ for every $M\in\om$. Fix an $M>0$. If $n$ is large enough, say $n\geq n_0$, then $\Phi(P_{\{k\}}(y-x_n))\leq B/M$ for every $k<M$ and hence \[\Phi(P_M(y))\leq \Phi(P_M(y-x_n))+\Phi(P_M(x_n))\leq \sum_{k<M}\Phi(P_{\{k\}}(y-x_n))+\Phi(x_n)\leq 2B.\]
Now we will prove that $x_n\to y$. If not, then there are $\eps>0$ and $n_0<n_1<\dots<n_j<\dots$ such that $\Phi(x_{n_j}-y)>\eps$, that is, $\Phi(P_{M_j}(x_{n_j}-y))>\eps$ for some $M_j\in \om\setminus\{0\}$ for every $j$. Pick $j_0$ such that  $\Phi(x_{n_{j_0}}-x_n)< \eps/2$ for every $n\geq n_{j_0}$ and then pick $j_1>j_0$ such that $\Phi(P_{\{k\}}(x_{n_{j_1}}-y))\leq \eps/(2 M_{j_0})$ for every $k<M_{j_0}$. Then
\[
\eps<\Phi(P_{M_{j_0}}(x_{n_{j_0}}-y))\leq \Phi(P_{M_{j_0}}(x_{n_{j_0}}-x_{n_{j_1}}))+\sum_{k<M_{j_0}}\Phi(P_{\{k\}}(x_{n_{j_1}}-y))< \eps,\]
a contradiction.

\smallskip
$\mrm{EXH}(\Phi)=\overline{c_{00}}$: $c_{00}$ is dense in $\mrm{EXH}(\Phi)$ because $\Phi(x-P_n(x))=\Phi(P_{\om\setminus n}(x))\xrightarrow{n\to\infty}0$ for every $x\in\mrm{EXH}(\Phi)$. We have to show that $\mrm{EXH}(\Phi)$ is closed. Let $x\in\mrm{FIN}(\Phi)$ be an accumulation point of $\mrm{EXH}(\Phi)$. For any $\eps>0$ we can find  $y\in\mrm{EXH}(\Phi)$ such that $\Phi(x-y)<\eps$, and then $n_0$ such that $\Phi(P_{\om\setminus n}(y))<\eps$ for every $n\geq n_0$. If $n\geq n_0$ then $\Phi(P_{\om\setminus n}(x))\leq \Phi(P_{\om\setminus n}(x-y))+\Phi(P_{\om\setminus n}(y))< 2\eps$.

\smallskip
$(e_n)$ is a $1$-unconditional basis in $\mrm{EXH}(\Phi)$: This follows from monotonicity of $\Phi$ (see the characterizations of unconditional bases in Section \ref{analPsec}).

\smallskip
Suppose $(b_n)$ is a $1$-unconditional basis in a Banach space  $X$. In particular, for each $x\in X$ there is a unique $(x_n)\in \mbb{R}^\om$ such that  $x = \sum_{n\in\om} x_n b_n$, and hence we can assume that $X\subseteq\mbb{R}^\om$. For
$x=(x_n)\in\mbb{R}^\om$ define $\Phi(x)=\sup\{\|\sum_{k<n}x_kb_k\|\colon n\in\om\}$. Applying $1$-unconditionality of $(b_n)$, $\Phi$ is a nice norm, and of course $\Phi(x)=\|x\|$ for every $x\in X$. Now, $x\in\mrm{EXH}(\Phi)$ iff
$\Phi(P_{\om\setminus n}(x))\xrightarrow{n\to\infty} 0$ iff $\sum_{n\in\om} x_nb_n$ is Cauchy/convergent iff $x\in X$.
\end{proof}

\begin{rem} Notice that $\mrm{tail}_\Phi(x)=\mrm{tail}_\Phi(y)$ whenever $x-y\in\mrm{EXH}(\Phi)$, in particular, we can consider $\mrm{tail}_\Phi\colon \mrm{FIN}(\Phi)/\mrm{EXH}(\Phi)\to [0,\infty)$, and it is straightforward to check that this map is the canonical  quotient norm on $\mrm{FIN}(\Phi)/\mrm{EXH}(\Phi)$.
\end{rem}

It is not a coincidence that when introducing the notions of $\mrm{FIN}(\Phi)$ and $\mrm{EXH}(\Phi)$ we mimicked the notations from the theory of analytic P-ideals. Every nice extended norm $\Phi$ gives rise in a natural way to many non-pathological  submeasures and vice versa:

\begin{prop}\label{re-presentation} Let $\Phi$ be a nice extended norm and
$\sigma=(\sigma_n)\in\mbb{R}^\om$. Then the function $\varphi(A)=\Phi(P_A(\sigma))$
is a non-pathological submeasure. Conversely, every non-pathological submeasure is of this form for some nice $\Phi$ and $\sigma=\mbf{1} = (1,1,\dots)$. In particular, every non-pathological analytic P-ideal $\mc{I}$ is represented in $\mrm{EXH}(\Phi)$ for
some nice $\Phi$ by the unconditional basis $(e_n)$, i.e. $\mathcal{I}=\mc{I}^{\mrm{EXH}(\Phi)}_{(e_n)}$.
\end{prop}
\begin{proof}
Lower semicontinuity of $\Phi$ and $1$-unconditionality of $(e_n)$ imply that $\varphi$ is an lsc submeasure. To show that it is non-pathological, notice that $\varphi=\varphi^{\mrm{EXH}(\Phi)}_{(\sigma_ne_n)}=\psi^{\mrm{EXH}(\Phi)}_{(\sigma_ne_n)}$
(see the proof of Theorem \ref{PAGBSrepr}, in particular $\mrm{Exh}(\varphi)=\mc{I}^{\mrm{EXH}(\Phi)}_{(\sigma_ne_n)}$).

\smallskip
Now let $\wt{\mu}=\{\mu_n\colon n\in\om\}$ be a family of measures and $\varphi_{\wt{\mu}}(A)=\sup\{\mu_n(A)\colon n\in\om\}$ be the associated non-pathological submeasure. Define $\Phi_{\wt{\mu}}\colon \mbb{R}^\om\to [0,\infty]$ as follows:
\[ \Phi_{\wt{\mu}}(x)=\sup\bigg\{\sum_{k\in\mrm{supp}(\mu_n)}\mu_n(\{k\})|x_k|\colon n\in\om\bigg\}.\]
Then $\Phi_{\wt{\mu}}$ is a nice extended norm and  $\Phi_{\wt{\mu}}(P_A(\mbf{1}))=\sup\{\mu_n(A)\colon n\in\om\}=\varphi_{\wt{\mu}}(A)$, and hence $\mrm{Exh}(\varphi_{\wt{\mu}})=\mc{I}^{\mrm{EXH}(\Phi_{\wt{\mu}})}_{(e_n)}$. Let us point out that we do not need to use $\varphi_{\wt{\mu}}=\psi^{\mrm{EXH}(\Phi_{\wt{\mu}})}_{(e_n)}$ to see that $\mrm{Exh}(\varphi_{\wt{\mu}})=\mc{I}^{\mrm{EXH}(\Phi_{\wt{\mu}})}_{(e_n)}$. Indeed, $A\in \mrm{Exh}(\varphi_{\wt{\mu}})$ iff $\varphi_{\wt{\mu}}(A \setminus n)=\Phi_{\wt{\mu}}(P_{A\setminus n}(\mbf{1}))\to 0$ iff $P_A(\mbf{1})\in\mrm{EXH}(\Phi_{\wt{\mu}})$ iff $\sum_{i\in A}e_i$ is unconditionally convergent.
\end{proof}

Also, notice that $(\sigma_n)$ and $(|\sigma_n|)$ generate the same submeasure, and we can assume, without changing the ideal, that $\sigma_n>0$ for every $n$ (i.e. that $\sigma\in (0,\infty)^\om$ just like $\tau$ in the definition of $\wt{\mu}(\mc{F},\tau$)).

\medskip
The next result was motivated by \cite[Theorem 3.2]{Ding17} and, in principle, it gathers some known results (Ding \cite{Ding17}, Drewnowski and Labuda \cite{Drewnowski10}, Bessaga-Pe\l czy\'{n}ski \cite{Bessaga58}) and put them in the framework of the notions and notations we introduced. For example, notice that (1)$\leftrightarrow$(2) below can be seen as a natural counterpart of Theorem \ref{char} (3).

\begin{thm}\label{characterization} Let $\Phi$ be a nice extended  norm. Then the following are equivalent:
\begin{enumerate}
\item $\mrm{FIN}(\Phi) = \mrm{EXH}(\Phi)$.
\item $\mrm{EXH}(\Phi)$ is $F_\sigma$ in $\mbb{R}^\om$.
\item $\mrm{EXH}(\Phi)$ does not contain $c_0$, i.e. $(e_n)$ is a boundedly complete basis.
\item $\mrm{FIN}(\Phi)$ is separable.
\item $\mbb{R}^\om/\mrm{EXH}(\Phi)$ is Borel reducible to $\mbb{R}^\om/\ell_\infty$.
\item Only $F_\sigma$ ideals can be represented in $\mrm{EXH}(\Phi)$.
\item $\mathcal{Z}$ is not represented in $\mrm{EXH}(\Phi)$.
\item If $\sigma\in\mbb{R}^\om$ and $\varphi(A)=\Phi(P_A(\sigma))$ then  $\mrm{Exh}(\varphi)=\mrm{Fin}(\varphi)$.
\end{enumerate}
\end{thm}
\begin{proof}
(1)$\to$(2): We show that $\{x\in\mbb{R}^\om\colon \Phi(x) \leq M\}$ is closed, and hence $\mrm{FIN}(\Phi)$ is $F_\sigma$ in $\mbb{R}^\om$. Indeed, assume that $\Phi(x)>M+\eps$ for some $\eps>0$. By lsc, there is $N\in\om$ such that $\Phi(P_N(x))>M+\eps$. As
	$\Phi$ is a norm on $c_{00}$, we can find $\delta>0$ such that $\Phi(P_N(y))<\eps$ whenever $|y(n)|<\delta$ for each $n\in N$. For such $y$, $\Phi(x+y)\geq \Phi(P_N(x+y))>M$. In other words, $\{x\in \mbb{R}^\om\colon \Phi(x)>M\}$ is open. (Similarly, $\mrm{EXH}(\Phi)$ is $F_{\sigma\delta}$.)

\smallskip
(2)$\to$(1): Let $\mrm{EXH}(\Phi) = \bigcup_{n\in\om} F_n$, where each $F_n$ is closed in the product topology. As norms on finite dimensional linear space are equivalent, the restriction of every basic open subset of $\mbb{R}^\om$ is open in
$(\mrm{EXH}(\Phi),\Phi)$ as well, and hence $F_n$ is closed in $(\mrm{EXH}(\Phi), \Phi)$ for every $n$. Applying the Baire Category Theorem, there are $n \in\omega$, $c\in\mrm{EXH}(\Phi)$, and $\varepsilon>0$ such that $B_c(\varepsilon) = \{x\in
	\mrm{EXH}(\Phi)\colon
\Phi(x-c)<\varepsilon\} \subseteq F_n$. Now if $y\in \mrm{FIN}(\Phi)\setminus\{0\}$, then $P_k(y)\in \mrm{EXH}(\Phi)$ and $\frac{\eps}{2\Phi(P_k(y))}P_k(y) +c \in B_c(\varepsilon) \subseteq F_n$ for every large enough $k$. Clearly, $P_k(y)\to y$ in the product topology, hence $\frac{\eps}{2\Phi(y)}y +c\in F_n\subseteq\mrm{EXH}(\Phi)$, and so $y\in \mrm{EXH}(\Phi)$.

\smallskip
(3)$\leftrightarrow$(1): $\mrm{FIN}(\Phi)=\mrm{EXH}(\Phi)$ says exactly that $(e_n)$ is boundedly complete.

\smallskip
(1)$\to$(4): Trivial.

\smallskip
(4)$\to$(1): If $x\in\mrm{FIN}(\Phi)\setminus\mrm{EXH}(\Phi)$, $\Phi(x)\geq \inf\{\Phi(P_{\om\setminus n}(x))\colon n\in\om\}>\eps>0$, then there are pairwise disjoint finite sets $F_n\subseteq\om$ such that $\Phi(P_{F_n}(x))>\eps$ for every $n$. For $A\subseteq\om$ let $F_A=\bigcup_{n\in A}F_n$ and $y_A=P_{F_A}(x)\in \mrm{FIN}(\Phi)$. If $A,B\subseteq\om$ and $n\in A\triangle B$ then $\Phi(y_A-y_B)\geq \Phi(P_{F_n}(x))>\eps$, and therefore no set of size less than continuum can be dense in $\mrm{FIN}(\Phi)$.

\smallskip
(1)$\leftrightarrow$(5): See \cite[Theorem 3.2]{Ding17}.

\smallskip
(3)$\leftrightarrow$(6): See Theorem \ref{DrLa}.

\smallskip
(6)$\to$(7): $\mathcal{Z}$ is not $F_\sigma$.

\smallskip
(7)$\to$(3): $\mathcal{Z}$ is representable in $c_0$.

\smallskip
(8)$\leftrightarrow$(1): If there is a $\sigma\in \mrm{FIN}(\Phi) \setminus \mrm{EXH}(\Phi)$ and $\varphi(A) = \Phi(P_A(\sigma))$, then $\varphi(\omega)=\Phi(\sigma)<\infty$ and so $\mrm{Fin}(\varphi)=\mathcal{P}(\omega)$ but $\omega\notin \mrm{Exh}(\varphi)$. Conversely, if $\sigma$ and $\varphi$ are as in (8) and $A\in \mrm{Fin}(\varphi) \setminus \mrm{Exh}(\varphi)$, then $P_A(\sigma) \in \mrm{FIN}(\Phi) \setminus \mrm{EXH}(\Phi)$.
\end{proof}

The above theorem raises several natural questions. One of them is whether $\mrm{EXH}(\Phi)$ not containing $\ell_1$ can be characterized in a way similar to the theorem above. Another is the question about the general relationship between $\mrm{EXH}(\Phi)$
and $\mrm{FIN}(\Phi)$. The following proposition partially solves both of these questions.

\begin{prop}\label{shrink} Let $\Phi$ be a nice extended norm. Then $\mrm{EXH}(\Phi)$ does not contain $\ell_1$ (i.e. $(e_n)$ is a shrinking basis) iff
$\mrm{FIN}(\Phi)$ is isomorphic to $\mrm{EXH}(\Phi)^{**}$.
\end{prop}

\begin{proof}
We know (see e.g. \cite[Proposition 3.3.6]{Schlumprecht}) that is $(e_n)$ is a shrinking basis in $X$ then $X^{**}$ is isomorphic to the space $\{\al=(\alpha_n)\in\mbb{R}^\om:
\Psi(\al)\colon =\sup\{\| \sum_{k<n} \alpha_k e_k \|:n\in\om\}<\infty\}=\mrm{FIN}(\Psi)$, and notice that $\Psi=\Phi$ in the case of $X=\mrm{EXH}(\Phi)$.

\smallskip
Conversely, if $\mrm{FIN}(\Phi)$ is isomorphic to $\mrm{EXH}(\Phi)^{**}$, then $|\mrm{EXH}(\Phi)^{**}|\leq \mathfrak{c}$, and by \cite[Main theorem]{Odell75}, $\mrm{EXH}(\Phi)$ does not contain $\ell_1$.
\end{proof}

It is however unclear for us if $\mrm{EXH}(\Phi)$ not containing $\ell_1$ could be characterized in terms of ideals represented in $\mrm{EXH}(\Phi)$ like in the case of $c_0$ in Theorem \ref{characterization}. Although for nice norms induced by
families of finite sets (see below) we have such characterization: $\mrm{EXH}(\Phi)$ does not contain a copy of $\ell_1$ if and only if no nontrivial $F_\sigma$ ideal can be represented in $\mrm{EXH}(\Phi)$ by its unconditional basis, see Theorem
\ref{compactly-supported}.

\medskip
In the rest of this section, we will present numerous examples of nice extended norms together with Banach spaces and analytic P-ideals induced by them. In fact all these examples will be of a very special form, generated by families of finite
subsets of $\omega$ (see also remarks after Theorem \ref{PAGBSrepr}). The
study of these families and the norms they induce have become an important section of Banach space theory in the past decades (see e.g. \cite{Argyros92}, \cite{Castillo93}, \cite{Jordi08},  \cite{Jordi15}, and \cite{Jordi-Stevo}).

Let $\mathcal{F} \subseteq [\om]^{<\om}$ be a cover of $\om$ containing all singletons. Define an extended norm $\Phi_\mathcal{F}$ by
\[ \Phi_\mathcal{F}(x) = \sup\bigg\{ \sum_{i\in F} |x_i|\colon F\in \mathcal{F}\bigg\}. \]
It is plain to check that it is a nice norm. Thus, $\mrm{EXH}(\Phi_\mathcal{F})$ and $\mrm{FIN}(\Phi_\mathcal{F})$ are Banach spaces. From now on, we will write $X_\mathcal{F} = \mrm{EXH}(\Phi_\mathcal{F})$.  Notice that we can always assume that $\mc{F}$ is hereditary (i.e. closed for taking subsets) because the hereditary closure $\mc{F}^\downarrow$ of $\mc{F}$ generates the same extended norm.

Now fix a sequence $\tau\in (0,\infty)^\om$ and consider a submeasure $\varphi_{\mathcal{F},\tau}$ defined from $\Phi_\mathcal{F}$ as before, i.e.
\[ \varphi_{\mathcal{F},\tau}(A) = \Phi_\mathcal{F}(P_A(\tau)). \]
Ideals of the form $\mathcal{I}_{\mathcal{F},\tau}$ will be called {\em $\mathcal{F}$-ideals}. Notice that trivial modifications of $\mrm{Fin}$ are $\mc{F}$-ideals for every $\mc{F}$. When we wish to exclude them we will talk about {\em
non-trivial} $\mc{F}$-ideals.

When studying classical examples, we will typically use one particular weight sequence. Let $\lambda =(\lambda_k)_{k\geq 1}$, $\lambda_k = 2^{-n}$ if $k\in [2^n, 2^{n+1})$, that is,
\[\lambda = \big(1,1/2,1/2,\underbrace{1/4,\cdots,1/4}_{4\;\text{times}},\underbrace{1/8,\cdots, 1/8}_{8\;\text{times}},\cdots\big).\]
Depending on the underlying set, we will modify $\lambda$ accordingly, on $\om\setminus n$ we consider the appropriate restriction of $\lambda$, and in the case of $2^{<\om}$ we consider $(\lambda_t)_{t\in 2^{<\om}}$, $\lambda_t=2^{-|t|}$; these
variants of $\lambda$ will also be denoted by $\lambda$.

\begin{exa} If $\mathcal{F} =[\om]^{\leq 1}$ then $\mrm{FIN}(\Phi_\mc{F})= \ell_\infty$, $\mrm{EXH}(\Phi_\mc{F})=c_0$, $\mc{I}_{\mc{F},\tau}=\{A\subseteq\om\colon A$ is finite or $(\tau_n)_{n\in A}\to 0\}$, and hence, by Fact \ref{nowheretall}, up to isomorphism the only nontrivial $\mc{F}$-ideal is $\{\0\}\otimes\mrm{Fin}$.
\end{exa}

\begin{exa} If $\mc{F}=[\om\setminus\{0\}]^{<\om}$ then
$\mrm{FIN}(\Phi_\mc{F})=\mrm{EXH}(\Phi_\mc{F})=\ell_1$ and, as $\mc{I}_{\mc{F},\tau}=\mc{I}_\tau$ (the summable ideal generated by $\tau$), the family of $\mc{F}$-ideals coincides with the family of all summable ideals. Notice that there are many
other families such that $\mc{I}_{1/n}$ is of the form $\mc{I}_{\mc{E},\tau}$, for example consider $\mathcal{E} = \{E\in [\omega\setminus\{0\}]^{<\omega}\colon  \sum_{n\in E}\lambda_n\leq 1\}$, then $\mc{I}_{1/n}=\mc{I}_{\mc{E},\lambda}$. At the same time,
$\mrm{EXH}(\Phi_\mathcal{E})$ contains a copy of $c_0$, e.g. because $\lambda$ witnesses that $\mrm{FIN}(\Phi_\mathcal{E})\ne\mrm{EXH}(\Phi_\mathcal{E})$.
\end{exa}

\begin{exa}\label{denszero} If $\mc{P}=\bigcup\{\mc{P}([2^n,2^{n+1}))\colon n\in\om\}$ then $\mrm{FIN}(\Phi_\mc{P})$ is isometrically isomorphic to the $\ell_\infty$-product of $(\ell_1(2^n)\colon n\in\om)$, similarly $\mrm{EXH}(\Phi_\mc{P})$ is isometrically isomorphic to $c_0$-product of this sequence. All $\mc{P}$-ideals are density ideals including the density zero ideal $\mc{Z}=\mc{I}_{\mc{P},\lam}$.
\end{exa}

\begin{exa} \label{trace-example} Let $\mc{A}=\{$finite antichains in $2^{<\om}\}$. Then $\mc{I}_{\mc{A},\lam}=\mrm{tr}(\mc{N})$ (see the discussion on non-pathological ideals after Theorem \ref{PAGBSrepr}). $\mrm{EXH}(\Phi_\mathcal{A})$ contains both copies of $c_0$ and $\ell_1$. Actually these copies are visible
	to the naked eye: copies of $c_0$ live on branches and copies of $\ell_1$ live on antichains, more precisely, $\{e_{f\clrest n}\colon n\in \omega\}$ is equivalent to the standard basis of $c_0$ for each $f\in 2^\omega$ and $\{e_s\colon s\in A\}$ is equivalent  to the
standard basis of $\ell_1$ for each infinite antichain $A\subseteq 2^{<\om}$.
	
Also, the norm $\Phi_\mathcal{A}$ is in fact a well-known norm on a certain subspace of $\mrm{FIN}(\Phi_\mathcal{A})$. Let $\mathcal{M}(2^\omega)$ be the space of signed Radon measures of bounded variation on $2^\omega$ equipped with the variation-norm, that is,
\begin{multline*}
	\|\mu\| =\mu^+(2^\om)+\mu^-(2^\om)=\\ \sup\big\{|\mu(C_0)| + |\mu(C_1)|\colon 2^\om=C_0\cup C_1\;\text{is a decomposition into clopen sets}\big\}.
\end{multline*}

We claim that $\mathcal{M}(2^\omega)$ can be isometrically embedded in $\mrm{FIN}(\Phi_\mathcal{A})$: For every $\mu\in \mathcal{M}(2^\omega)$ let $x_\mu\colon 2^{<\omega}\to\mbb{R}$, $x_\mu(t) = \mu([t])$ where $[t]=\{f\in 2^\om\colon t\subseteq f\}$ is the basic clopen set generated by $t$. Then the function $\mu\mapsto x_\mu$ is linear and injective (because $\mu$ is uniquely determined by its values on basic clopen sets). We show that $\|\mu\|=\Phi_\mc{A}(x_\mu)$, and hence this is an isometric embedding $\mc{M}(2^\om)\to\mrm{FIN}(\Phi_\mc{A})$.

Let $\mu\in \mathcal{M}(2^\omega)$, $\varepsilon>0$, and $C_0$, $C_1$ be disjoint clopen sets such that $2^\om=C_0\cup C_1$ and $|\mu(C_0)|+|\mu(C_1)|>\|\mu\|-\eps$. There are $A_0,A_1\in\mc{A}$ such that $C_i=\bigcup_{t\in A_i} [t]$, notice that
$A_0$ and $A_1$ must be disjoint and $A:=A_0\cup A_1$ is also an antichain.  Now, by definition $\Phi_\mathcal{A}(x_\mu)\geq \sum_{t\in A}|x_\mu(t)|=\sum_{t\in A_0}|\mu([t])|+\sum_{t\in A_1}|\mu([t])|\geq |\mu(C_0)| + |\mu(C_1)| > \|\mu\|-\varepsilon$. The same argument works in the other direction as well.

We do not know if $\mrm{EXH}(\Phi_\mc{A})$ (or $\mrm{FIN}(\Phi_\mc{A})$) is isomorphic to a known Banach space.
\end{exa}

\section{Spaces and ideals generated by compact families}\label{cptsupp}

In the rest of the paper we will mostly discuss properties of $X_\mc{F}$ and $\mc{F}$-ideals. First of all, in this section, we will focus on precompact families of finite sets. We say that an $\mc{F}\subseteq [\omega]^{<\omega}$ is {\em precompact} if
$\overline{\mc{F}}\subseteq [\omega]^{<\omega}$, or equivalently, if every sequence $(F_n)$ in $\mc{F}$ has a subsequence which forms a $\Delta$-system (for detailed studies on precompact families, see e.g. \cite{Jordi15} and \cite{Jordi-Stevo}). The name is motivated by the fact that an $\mc{F}\subseteq [\omega]^{<\omega}$ is compact iff every sequence $(F_n)$
in $\mc{F}$ has a subsequence which forms a $\Delta$-system with root from $\mc{F}$. We will leave the following simple fact without proof.

\begin{fact}
	If $\mc{H}\subseteq\mc{P}(\om)$, then $\overline{\mc{H}}^\downarrow=\overline{\mc{H}^\downarrow}$. In particular, (a) if $\mc{H}$ is closed then so is $\mc{H}^\downarrow$; and (b) if $\mc{F}\subseteq [\omega]^{<\omega}$ then
	$\overline{\mc{F}}\subseteq [\omega]^{<\omega}$ iff $\mc{F}^\downarrow$ does not contain infinite chains.
\end{fact}

Notice that if $\mc{F}\subseteq[\om]^{<\om}$, then $\mc{F}$, $\mc{F}^\downarrow$,  $\overline{\mc{F}}\cap [\om]^{<\om}$ and $\overline{\mc{F}}^\downarrow\cap[\om]^{<\om} = \overline{\mc{F}^\downarrow}\cap [\om]^{<\om}$ generate the same space $X_\mc{F}$.

\begin{exa}\label{locally-finite}
Probably the most obvious examples of compact families of finite sets are the locally finite ones, that is, when $\{F\in\mc{F}\colon n\in F\}$ is finite for every $n$ (and non-empty, as we always work with covers). It follows from \cite[Proposition 6.4]{Our-paper} that in this case all $\mc{F}$-ideals are generalized density ideals.
\end{exa}

Before the main result of this section, we recall some notions and results from the theory of Banach spaces and also from general topology:

If $X,Y$ are Banach spaces then we say that $X$ is {\em $Y$-saturated} if every infinite dimensional closed subspace of $X$ contains a copy of $Y$.

If $X$ is topological space, then by recursion on $\al$ we define the {\em Cantor-Bendixson (CB) derivatives} of $X$ as follows: $X^{(0)}=X$, $X'=\{$accumulation points of $X\}$,  $X^{(\al+1)}=(X^{(\al)})'$, and if $\ga$ is limit then
$X^{(\ga)}=\bigcap_{\al<\ga}X^{(\al)}$. Let \[ \mrm{rk}(X)=\min\{\al\colon X^{(\al)}=X^{(\al+1)}\} \] denote the {\em CB-rank} of $X$. If $X$ is countable and compact, then $\mrm{rk}(X)$ is always a successor ordinal, and if  $\mrm{rk}(X)=\al+1$, then
$X^{(\al)}$ is (nonempty) finite (and so $X^{(\al+1)}=\0$). The classical theorem of Mazurkiewicz and Sierpi\'{n}ski says (see \cite{Sierpinski}) that in this case, if $|X^{(\al)}|=M$, then $X$ is homeomorphic to the ordinal $\om^\al M+1$ (equipped with the topology generated by its natural well-order).

We will use the following easy observation: If $X$ and $Y$ are compact metric spaces and $f\colon X\to Y$ is a continuous open surjection with finite fibers (i.e. $|f^{-1}(\{y\})|<\om$ for every $y\in Y$), then $X^{(\al)}=f^{-1}[Y^{(\al)}]$ for every $\al$. In particular, in this case $\mrm{rk}(X)=\mrm{rk}(Y)$.

\begin{thm}\label{compactly-supported} Let $\mc{F}\subseteq [\om]^{<\om}$ be a cover of $\om$. Then the following are equivalent:
\begin{itemize}
\item[(1)] $\mc{F}$ is precompact.
\item[(2)] $X_\mc{F}$ is $c_0$-saturated.
\item[(3)] $X_\mc{F}$ does not contain $\ell_1$, i.e. $(e_n)$ is a shrinking basis.
\item[(4)] Non-trivial $\mc{F}$-ideals are not $F_\sigma$.
\end{itemize}
\end{thm}
\begin{proof}
(1)$\to$(2): As $X_\mc{F}=X_{\overline{\mc{F}}}$, we may assume that $\mc{F}$ is compact. We know (see \cite{Castillo-Gonzales}) that if $\mc{F}$ is compact then $X_\mc{F}$ can be embedded isometrically in a $C(K)$ space, more precisely, in
$C(K_\mc{F})$ where
\[K_\mc{F} =\big\{y=(y_n)\in \{-1,0,1\}^\om\colon \mrm{supp}(y)\in \mathcal{F}\big\}.\] It is trivial to check that $K_\mc{F}$ is compact, and that the map $X_\mc{F}\to C(K_\mc{F})$, $x=(x_n)\mapsto f_x$, $f_x(y)=\sum_{n\in \mrm{supp}(y)}x_ny_n$ is an isometry.
Since Banach spaces of the form $C(K)$, where $K$ is countable and compact, are $c_0$-saturated (see \cite{Semadeni59}), $X_\mc{F}$ is also $c_0$-saturated.

\smallskip
(2)$\to$(3): Trivial.

\smallskip
(3)$\to$(1) and (4)$\to$(1): Suppose now that $\mathcal{F}$ is not precompact. Then, since we can assume that $\mathcal{F}$ is hereditary (because $X_\mc{F}=X_{\mc{F}^\downarrow}$), there is an infinite $A\subseteq \omega$ such that $A\cap n \in
\mathcal{F}$ for each $n$. As $\Phi_\mc{F}(P_A(x))=\sup\{\sum_{i\in A\cap n} |x_i|\colon n\in \omega\}=\sum_{i \in A} |x_i|$, the sequence $(e_i)_{i\in A}$ is isometrically equivalent to the standard basis of $\ell_1$. In particular, if $\tau=(\tau_n)\in (0,\infty)^\om$ such that $\tau_n\to 0$, $\sum_{n\in A}\tau_n=\infty$, and $\sum_{n\notin A}\tau_n<\infty$, then $\mc{I}_{\mc{F},\tau}$ is a tall summable ideal.

\smallskip
(1)$\to$(4): As above, we may assume that $\mathcal{F}$ is compact. Now, suppose on the contrary that there is an $F_\sigma$ P-ideal $\mc{I}$ which is not a trivial modification of $\mrm{Fin}$ but $\mc{I}=\mc{I}_{\mc{F},\tau}$ for some compact $\mc{F}$.
Applying Fact \ref{nowheretall} and Corollary \ref{nowheretallcor}, there is a $D\in\mc{I}^+$ such that $\mc{I}\clrest D$ is tall. Now, $\mc{F}[D]=\{D\cap F\colon F\in \mc{F}\}$ is compact in $\mc{P}(D)$, it covers $D$, and $\mc{I}\clrest D=\mc{I}_{\mc{F}[D],\tau\clrest D}$. Therefore, we can assume that our counterexample is tall (from now on $D=\om$, of course): $\mathcal{I}=\mathrm{Exh}(\varphi)=\mathrm{Fin}(\varphi)$ for some lsc $\varphi$ such that $\varphi(\{n\})\to 0$.

\smallskip
Let $\al$ be the minimal ordinal such that there are a compact $\mc{F}\subseteq [\om]^{<\om}$ of rank $\al+1$ covering $\om$ and a $\tau\in (0,\infty)^\om$ such that $\mc{I}=\mc{I}_{\mc{F},\tau}$. Fix $\mc{F}$ and $\tau$ witnessing this
definition.

\begin{claim}
We may assume that $\mathcal{F}^{(\alpha)} = \{\emptyset\}$.
\end{claim}
\begin{proof}[Proof of the claim] 
We know that $\mathcal{F}^{(\alpha)}$ is finite. Let $N>\max\{\max(F)\colon F\in \mathcal{F}^{(\alpha)}\}$ and $\mathcal{F}[\om\setminus N]= \{F\setminus N\colon F\in \mathcal{F}\}$. Since $\mathcal{F} \to \mc{F}[\omega\setminus N]$, $F\mapsto F \setminus N$ is continuous, open, and has finite fibers, $\mrm{rk}(\mathcal{F}[\om\setminus N]) = \mrm{rk}(\mathcal{F})$.  Define the family $\mathcal{G} = \mathcal{F}[\om\setminus N]\cup\{ \{n\}\colon n\leq N\}$. Then $\mrm{rk}(\mc{G})=\mrm{rk}(\mathcal{F})$, $\mathcal{G}^{(\alpha)}=\{\0\}$, and $\mathcal{I}_{\mathcal{F}, \tau} = \mathcal{I}_{\mathcal{G},\tau}$, hence we can work with $\mathcal{G}$ instead of $\mathcal{F}$.
\end{proof}

Clearly, $\al\geq 1$. In fact $\al\geq 2$. If $\al=1$ then $\mc{F}$ is locally finite, hence $\mc{I}$ is a tall generalized density ideal (see Example \ref{locally-finite}), which cannot be an $F_\sigma$ ideal.

\smallskip
The map $K_\mc{F}\to \mc{F}$, $y\mapsto\mrm{supp}(y)$ is a continuous open surjection with finite fibers, hence $K_\mc{F}$  is also of rank $\al+1$, $K_\mc{F}^{(\al)}=\{$constant $0$ sequence$\}$, and so there is a homeomorphism $\omega^\alpha+1\to K_\mathcal{F}$, $\beta \mapsto y^\beta$ such that
$y^{\om^\alpha}$ is the constant $0$ sequence.

\smallskip
Let $f_n\in C(K_\mc{F})$ be the image of $e_n\in X_\mc{F}$ under the isometric embedding $X_\mc{F}\to C(K_\mc{F})$, that is, $f_n(y)=y_n$.
As $f_n$ is continuous, $\mrm{ran}(f_n)\subseteq\{-1,0,1\}$, and $f_n(y^{\om^\al})=0$, we can fix $\ga_n<\om^\al$ such that $f_n(y^{\delta})=0$ for every $\delta\in (\ga_n,\om^\al]$. We can assume that $\ga_0<\ga_1<\cdots$ tends to $\om^\al$.

For $n\in\om$ let \[ K_n=\big\{y^{\be}\colon \be\leq\ga_n\big\},\] a compact subset of $K_\mc{F}$.
Clearly $\bigcup_{n\in\om}K_n=K_\mc{F}$. Define the lsc submeasures $\psi$ and $\psi_n$ for $n\in\om$ on $\mc{P}(\om)$ as follows:
\begin{align*}
\psi(A)=&\sup\bigg\{\bigg\|\sum_{i\in B} \tau_if_i\bigg\|\colon B\in [A]^{<\om}\bigg\},\\
\psi_n(A)=&\sup\bigg\{\bigg\|\sum_{i\in B}\tau_if_i\clrest K_n\bigg\|\colon B\in [A]^{<\om}\bigg\}
\end{align*}
where $\tau_if_i\clrest K_n\in C(K_n)$ and $\|\cdot\|$ denotes the sup-norm on these spaces. Then $\mrm{Exh}(\psi)=\mc{I}$ and hence we can assume that $\psi \leq \varphi$ on $\mc{P}(\om)$ (otherwise we can work with $\varphi' = \psi + \varphi$ because $\mc{I}=\mrm{Exh}(\varphi')=\mrm{Fin}(\varphi')$). Let
\[ \mathcal{I}_n = \mathrm{Exh}(\psi_n)=
\Big\{A\subseteq \omega \colon \sum_{i\in A}\tau_if_i\clrest K_n\;\text{is unconditionally convergent in}\;C(K_n)\Big\}.\]
Notice that $(\mathcal{I}_n)_{n\in\om}$ forms a non-increasing sequence of ideals containing $\mathcal{I}$.

\begin{claim}
	$\mc{I}\ne \mc{I}_n$ for every $n\in\om$.
\end{claim}
\begin{proof}[Proof of the claim] Let $\mc{G}=\{\mrm{supp}(y)\colon y\in K_n\}$. Then $\mc{G}$ is a continuous open image of $K_n$ with finite fibers, and hence it is compact and $\mrm{rk}(\mc{G})=\mrm{rk}(K_n)< \al+1$. Notice that $K_n$ is a subset
	of $K_\mc{G}=\{y\in\{-1,0,1\}^\om\colon \mrm{supp}(y)\in \mc{G}\}$ and typically it is a proper subset. The family $\mc{G}$ does not necessarily cover $\om$, so define $G=\bigcup\mc{G}$ and $\mc{H}=\mc{G}\cup[\om\setminus G]^{\leq 1}$. Then $\mc{H}$ is compact and $\mrm{rk}(\mc{G})\leq \mrm{rk}(\mc{H})\leq\max\{\mrm{rk}(\mc{G}),2\}< \al+1$ (because $\al\geq 2$).

We show that $\mc{I}_n=\mc{I}_{\mc{H},\sigma}$ where $\sigma\clrest G=\tau\clrest G$ and $\sigma_i=2^{-i}$ for $i\in \om\setminus G$, and hence $\mc{I}_n\ne\mc{I}$ by the definition of $\al$.

$K_n\subseteq K_\mc{H}$ and if $g_i$ is the image of $e_i\in X_\mc{H}\to C(K_\mc{H})$, then $g_i\clrest K_n=f_i\clrest K_n$, in particular, if $\sum_{i\in A}\sigma_ig_i$ is unconditionally convergent in $C(K_\mc{H})$, then $\sum_{i\in A}\sigma_i f_i\clrest K_n$ is unconditionally convergent in $C(K_n)$. As $f_i(y)=0$ for every $i\in(\om\setminus G)$ and $y\in K_n$, $\sum_{i\in A}\tau_if_i\clrest K_n$ is also unconditionally convergent, therefore $\mc{I}_{\mc{H},\sigma}\subseteq\mc{I}_n$.

Conversely, assume that $\sum_{i\in A}\tau_if_i\clrest K_n$ is unconditionally Cauchy, that is, \[ \forall\;\eps>0\;\exists\;N_\eps\in\om\;\forall\;B\in [A\setminus N_\eps]^{<\om}\;\bigg\|\sum_{i\in B}\tau_if_i\clrest K_n\bigg\|<\eps.\] We know that  $\|\sum_{i\in B}\tau_if_i\clrest K_n\|<\eps$ holds iff $|\sum_{i\in B}\tau_if_i(y)|=|\sum_{i\in B}\tau_iy_i|<\eps$ for every $y\in K_n$. We would like to show that $\sum_{i\in A}\sigma_ig_i$ is also unconditionally Cauchy in $K_\mc{H}$. Fix $\eps>0$, let $N\geq N_{\eps/2}$ such that $\sum_{i\in \om\setminus (G\cup N)}2^{-i}<\eps$, finally let $B\in [A\setminus N]^{<\om}$ and $y\in K_\mc{H}$ be arbitrary.

If $\mrm{supp}(y)\notin\mc{G}$ then \[ \bigg|\sum_{i\in B}\sigma_ig_i(y)\bigg|\leq \bigg|\sum_{i\in B\cap G}\sigma_i\cdot 0\bigg|+\bigg|\sum_{i\in B\setminus G}\sigma_iy_i\bigg|\leq \sum_{i\in B\setminus G}2^{-i}<\eps.\]

Now assume that $\mrm{supp}(y)\in \mc{G}$ but $y\notin K_n$ (for $y\in K_n$ we are done) and denote $B' = B \cap \mrm{supp}(y)\subseteq G$. Then $|\sum_{i\in B}\sigma_iy_i|=|\sum_{i\in B'}\sigma_iy_i|=|\sum_{i\in
B'}\tau_iy_i|$. We know that there is $y'\in K_n$ with the same support. Now if we partition $B'=P\cup Q$ where $P=\{i\colon y_iy'_i=1\}$ and $Q=\{i\colon y_iy'_i=-1\}$, then \[ \bigg|\sum_{i\in B'}\tau_iy_i\bigg|\leq\bigg|\sum_{i\in P}\tau_iy_i\bigg|+\bigg|\sum_{i\in
Q}\tau_iy_i\bigg|=\bigg|\sum_{i\in P}\tau_iy'_i\bigg|+\bigg|\sum_{i\in Q}\tau_iy'_i\bigg|<\eps.\qedhere\]
\end{proof}

If $E,F\in [\om]^{<\om}$ are non-empty then we write $E\leq F$ if $\max(E)\leq \min(F)$, and similarly, $E< F$ if $\max(E)<\min(F)$.

\begin{claim} There is a sequence $(A_n)_{n\in\om}$ such that (i) $A_n \in \mc{I}_n\setminus \mc{I}$ for every $n$, (ii) $\psi_n(A_n) < 2^{-n}$ for every $n$, and (iii) $A_n\cap A_k=\0$ if $n\ne k$.
\end{claim}
\begin{proof}[Proof of the Claim]
Fix an arbitrary sequence $(I_n)_{n\in\om}$ such that (i) holds. Let $(J_k)$ be a sequence such that $J_k = I_n$ for each $k$ and some $n$, and $H_n = \{k\in\om\colon J_k = I_n\}$ is infinite for every $n\in\om$. By recursion on $k$ we can construct
a sequence $B_0<B_1<\cdots$ of finite sets such that $B_k\subseteq J_k$ and $\varphi(B_k)>k$ (because $J_k\notin\mc{I}=\mrm{Fin}(\varphi)$). Now the sequence  $A_n = \bigcup_{k\in H_n}B_k$ satisfies (i) and (iii). By throwing out finitely many points from each $A_n$ we obtain the desired family.
\end{proof}

In particular, since $\psi_n \leq \psi_k$ for $n\leq k$, we know that $A_{\geq n}\colon =\bigcup_{k\geq n} A_k \in \mathcal{I}_n$ for every $n$, and by (ii) that $\psi_n(A_{\geq n})<2^{-n+1}$.

\smallskip
We can easily construct a sequence $X'_{0}<X'_{1}<\cdots$ of non-empty finite subsets $\om$ such that
$\varphi(X'_n) \approx 1$, $X'_{0} \subseteq A_{0}$, and $X'_{n+1} \subseteq A_{m_n}$ where $m_n=\max(X'_n)+1$ for every $n$.

\begin{claim} There is a set $X \subseteq \bigcup_n X'_n$ such that $\varphi(X \cap X'_n)\to 0$ and $\varphi(X)=\infty$.
\end{claim}
\begin{proof}[Proof fo the claim]
Let $X'=\bigcup_{n\in\om}X'_n$. Consider the ideals $\mathcal{I}' =\mathcal{I}\clrest X'$ and $\mathcal{Z}_{\vec{\vartheta}}$, where $\vec{\vartheta} = (\vartheta_n)$, $\vartheta_n = \varphi\clrest X'_n$. Clearly $\mathcal{I}'\subseteq
\mathcal{Z}_{\vec{\vartheta}}$ but since $\mathcal{I}'$ is an $F_\sigma$ ideal and $\mathcal{Z}_{\vec{\vartheta}}$ is a  tall generalized density ideal, they are not equal and every $X\in \mathcal{Z}_{\vec{\vartheta}}\setminus \mathcal{I}'$ is
as desired.
\end{proof}

Finally, we show that $X\in \mathrm{Exh}(\psi)$, and hence $\mrm{Exh}(\psi)\ne\mrm{Fin}(\varphi)$, a contradiction.

Let $X_n = X \cap X'_n$. Notice that $X\setminus m_n\subseteq A_{\geq m_n}$ because $X_{n+1}\subseteq A_{m_n}$, $X_{n+2}\subseteq A_{m_{n+1}}$ etc. Now $\psi=\sup\{\psi_y\colon y\in K_\mc{F}\}$ where
\[ \psi_y(A)=\sup\bigg\{\bigg|\sum_{i\in B}\tau_if_i(y)\bigg|\colon B\in [A]^{<\om}\bigg\}\]
is an lsc submeasure on $\om$. Fix such $y$ and the smallest $m$ such that $y\in K_m$.

If $m\leq m_n$ then $\psi_y(X\setminus m_n)\leq\psi_{m_n}(X\setminus m_n)\leq\psi_{m_n}(A_{\geq m_n})< 2^{-m_n+1}\leq 2^{-n}$.

If $m_n\leq m_{n'}\leq m < m_{n'+1}$ for some $n'\geq n$ then
\[ \psi_y\big(X\setminus m_n\big)\leq \psi_y\big(X\cap [m_n,m_{n'})\big)+\psi_y\big(X\cap [m_{n'},m_{n'+1})\big)+\psi_y\big(X\setminus m_{n'+1}\big)\]
where $\psi_y(X\cap [m_n,m_{n'}))=0$ because $y=y^\de$ for some $\de>\ga_{m-1}\geq \ga_{m_{n'}-1}$ and so $f_i(y)=0$ for every $i<m_{n'}$; as $X_{n'+1}=X\cap [m_{n'},m_{n'+1})$ we know that \[ \psi_y(X\cap [m_{n'},m_{n'+1}))=\psi_y(X_{n'+1})\leq\psi(X_{n'+1})\leq \varphi(X_{n'+1})\] and $\psi_y(X\setminus m_{n'+1})<2^{-n'-1+1}\leq 2^{-n}$ just like in the first case.

Therefore, $\psi_y(X\setminus m_n)\leq \varphi(X_{n(y)})+2^{-n}$ for every $y$ with some $n(y)>n$. As $\varphi(X_n)\xrightarrow{n\to\infty} 0$, we obtain that $\psi(X\setminus m_n)\xrightarrow{n\to\infty}0$ as well.
\end{proof}

\begin{rem}
In the same way as above one can prove a strengthening of Theorem \ref{our-theorem}. Namely, if $\mathcal{I}$ is represented in $C(\ga+1)$, for some countable ordinal $\ga$, then $\mathcal{I}$ is either a summable ideal or not $F_\sigma$: let $\ga$ be minimal such that there is an $F_\sigma$ ideal $\mathcal{I}$ which is not summable but represented in $C(\ga+1)$. Since $\mathcal{I}$ is not summable, $\ga$ has to be infinite and then it is easy to see that it has to be a limit ordinal.

Since $C(\ga+1)$ is isomorphic to $C_0(\ga+1)=\{f\in C(\al+1)\colon f(\ga)=0\}$ we may assume that $\mathcal{I}$ is represented in $C_0(\ga+1)$. If $(f_n)$ is a sequence in $C_0(\ga+1)$ representing $\mathcal{I}$, then we may assume that there is an
increasing sequence $\ga_n$ of ordinals converging to $\ga$ such that $f_n(\delta)=0$ for $\delta>\ga_n$. For each $n$ define $\mathcal{I}_n$ as the ideal represented by the sequence $(f_n\clrest (\ga_n+1))$ in $C(\ga_n+1)$ and notice that
$(\mathcal{I}_n)$ is non-increasing, every $\mathcal{I}_n$ contains $\mathcal{I}$, and $\mathcal{I} \ne \mathcal{I}_n$ because of the choice of $\ga$. From this point on we can continue the proof as above, starting with Claim C.

Since every space of the form $X_\mathcal{F}$, for a precompact $\mathcal{F}$, can be embedded in $C(\ga+1)$ for some $\ga$, as a corollary we obtain that among $F_\sigma$ ideals only summable ideals can be represented in $X_\mathcal{F}$ for a
precompact $\mathcal{F}$.
\end{rem}

\begin{prob}
Is $\mrm{tr}(\mc{N})$ an $\mc{F}$-ideal for some (pre)compact $\mc{F}$?
\end{prob}

\section{Applications to Pt\'{a}k's Lemma and Mazur's Lemma}\label{Ptak}

In this section we take a short digression to show that Theorem \ref{compactly-supported} has nice applications in combinatorics. Consider the following problem.

\begin{prob}\label{du} Fix a measure $\mu\colon \mathcal{P}(\omega) \to [0,\infty]$. Is there a compact hereditary family $\mathcal{F}$ of finite subsets of $\omega$ such that for every $E\in [\omega]^{<\omega}$ there is $F\in \mathcal{F}$ such that
	$F\subseteq E$ and $\mu(F)\geq \mu(E)/2$?
\end{prob}

This is a version of of Fremlin's DU Problem (see \cite{Fremlin-DU}, although Fremlin's note contains many versions of DU problem but this one). Notice that if $\mu$ is a counting measure, then the answer is positive, witnessed by the Schreier
family. However, if $\mu$ is a measure generating a tall proper ideal, that is, we assume that $\mu(\om)=\infty$ and that $\mu(\{n\})\to 0$, then the
answer seems to be difficult to guess. Theorem \ref{compactly-supported} implies that in this situation it is negative.

\begin{thm}\label{filling} Let  $\mu\colon \mathcal{P}(\omega) \to [0,\infty)$ be a measure such that $\mu(\omega)=\infty$ and $\mu(\{n\})\to 0$, $\mathcal{F}$ be a hereditary family of finite subsets of $\omega$, and fix $\varepsilon>0$. If for every $E\in [\omega]^{<\omega}$ there is $F\in \mathcal{F}$ such that $F\subseteq E$ and $\mu(F)\geq \varepsilon \mu(E)$, then $\mathcal{F}$ is not compact.
\end{thm}
\begin{proof} Assume $\mathcal{F}$ is as in the theorem, consider the lsc submeasure defined by \[ \varphi(A) = \sup\{\mu(F\cap A)\colon F\in \mathcal{F}\},\] let $\mathcal{I} = \mathrm{Exh}(\varphi)$ and $\mathcal{J} = \mathrm{Exh}(\mu)$. As $\varphi \leq \mu$, we  know that $\mathcal{J}\subseteq \mathcal{I}$. On the other hand, $\varphi(A\setminus n) \geq \varepsilon \mu(A\setminus n)$ for every $A\subseteq\omega$ and $n\in \omega$, hence $\mathcal{I}\subseteq \mathcal{J}$ also holds, and so $\mathcal{I}=\mathcal{J}$. But $\mathcal{J}$ is a tall $F_\sigma$ ideal and $\mathcal{I}$ is and $\mc{F}$-ideal, a contradiction with Theorem \ref{compactly-supported}.
\end{proof}

This proof seems to be rather indirect. However, it is unclear for us how to prove it more directly. Note that most of the proofs of the negative answers to variants of the DU Problem uses  precalibers of measures.

Theorem \ref{filling} seems to be surprising if compared to Pt\'{a}k's Lemma (see \cite{Ptak}). Denote by $\mathfrak{M}$ the set of finitely supported probability measures on $\omega$.

\begin{thm}[Pt\'ak's Lemma] Let $\mathcal{F}$ be a hereditary family of finite subsets of $\omega$, and fix $\eps>0$. If for each $\nu\in \mathfrak{M}$ there is $F\in \mathcal{F}$ such that $\nu(F)\geq \varepsilon$, then
$\mathcal{F}$ is not compact.
\end{thm}

Using Theorem \ref{filling} we can see that instead of testing $\mathcal{F}$ against every measure in $\mathfrak{M}$, we may focus on finitely supported measures generated by just one measure (see also \cite[Remark II.3.33]{Ramsey-methods}). For a measure $\mu$ on $\omega$ denote $\mathfrak{M}_\mu$ the following family of finitely
supported probability measures on $\omega$: $\nu\in\mf{M}_\mu$ iff $\mrm{supp}(\nu)\in [\om]^{<\om}$, $\mu(\mrm{supp}(\nu))>0$, and
$\nu(A) = \mu(A\cap\mrm{supp}(\nu))/\mu(\mathrm{supp}(\nu))$ for $A\subseteq\om$.

\begin{cor}\label{localized-Ptak} Let  $\mu$ be a measure on $\omega$ such that $\mu(\omega)=\infty$ and $\mu(\{n\})\to 0$, $\mathcal{F}$ be a hereditary family of finite subsets of $\omega$, and fix $\varepsilon>0$. If for each $\nu \in\mathfrak{M}_\mu$ there is an $F\in\mc{F}$ such that  $\nu(F)\geq\varepsilon$, then $\mathcal{F}$ is not compact.
\end{cor}
\begin{proof} This is just a reformulation of Theorem \ref{filling}.
\end{proof}

\begin{thm}[Mazur's Lemma] Let $X$ be a Banach space and let $(x_n)$ be a bounded weakly null sequence in $X$. Then for each $\varepsilon>0$ there is a finite convex combination $y = \sum_{i} \al_i x_i$ such that $\|y\|<\varepsilon$.
\end{thm}

Mazur's Lemma can be derived directly from Pt\'ak's Lemma (see \cite[Corollary II.3.35]{Ramsey-methods} for the detailed proof): For $x^*\in X^*$ define $F_{x^*} = \{n\in \omega\colon |x^*(x_n)|\geq \varepsilon/2\}$, and let $\mathcal{F}$ be the hereditary closure of $\{F_{x^*}\colon x^*\in X^*\}$. Notice that $\mathcal{F}$ is compact. Fix $\varepsilon>0$. Using Pt\'ak's Lemma, one can find $\nu\in \mathfrak{M}$ such that $\nu(F)<\varepsilon$ for some $F\in \mathcal{F}$, and then $y = \sum_{i\in \mathrm{supp}(\nu)} \nu(\{i\}) x_i$ is as desired.

Using Corollary \ref{localized-Ptak} instead of Pt\'ak's Lemma in the above proof, we may specify the form of the convex combination in Mazur's Lemma:

\begin{cor} Let $X$ be a Banach space, $(x_n)$ be a bounded weakly null sequence in $X$, and let $\mu$ be a measure on $\omega$ such that $\mu(\omega)=\infty$ and $\mu(\{n\})\to 0$. Then for each $\varepsilon>0$ there is a
	finite $G\subseteq \omega$ and a convex combination $y = \sum_{i\in G} \al_i x_i$ where $\al_i = \mu(\{i\})/\mu(G)$, such that $\|y\|<\varepsilon$.
\end{cor}

\section{Schreier ideals}\label{schreier}
In Example \ref{trace-example} we already saw an example of a family of finite sets which induces a well-known analytic P-ideal, and at the same time generates a Banach space whose properties are quite unclear and which may be, at least hypothetically, a new and interesting example. This
section is devoted to an opposite situation. We will define a family of analytic P-ideals, which, to the best of our knowledge, were not described before in the literature and which are induced by the so called Schreier families, well-known in the theory
of Banach spaces (see e.g. \cite{Castillo-Gonzales}, \cite{Gasparis-Leung}).

We begin with recalling the definition of {\em Schreier families} (introduced in \cite{Alspach-Argyros}). Fix a {\em ladder system} \[ \vec\xi=(\xi^\al_n\colon \al<\om_1 \mbox{ is a limit ordinal and }n\in\om\setminus\{0\}) \] on $\om_1$, that is,
$\xi^\al_1<\xi^\al_2<\cdots$ tends to $\al$ for every limit $\al$; and by recursion on $\al<\om_1$ define $\mc{S}_\al=\mc{S}^{\vec\xi}_\al\subseteq [\om\setminus\{0\}]^{<\om}$ as follows: $\mc{S}_0=[\om\setminus \{0\}]^{\leq 1}$,
\[ \mc{S}_{\al+1}=\{\0\}\cup\left\{\bigcup_{j=1}^nE_j\colon n\in\om\setminus\{0\},\;E_j\in \mc{S}_\al\setminus\{\0\},\;\{n\}\leq E_1<E_2<\cdots< E_n\right\},\]
and for limit $\al$ let
\begin{align*}
\mc{S}_\al&=\{\0\}\cup\big\{E\colon \exists\;k\in\om\setminus\{0\}\;\big(\{k\}\leq E\;\text{and}\;E\in \mc{S}_{\xi^\al_k}\setminus\{\0\}\big)\big\}\\
&=\{\0\}\cup\bigcup_{k=1}^\infty\mc{S}_{\xi^\al_k}\clrest(\om\setminus k)\end{align*}
where if $\mc{F}\subseteq [\om\setminus \{\0\}]^{<\om}$ and $A\subseteq\om$ then $\mc{F}\clrest A=\{F\in\mc{F}\colon F\subseteq A\}$.

Notice that the classical Schreier family $\mc{S}=\mc{S}_1=\{\0\}\cup \{F\subseteq \om\setminus\{0\}\colon |F|\leq\min(F)\}$. For each $\alpha<\omega_1$ the family $\mc{S}_\al$ is hereditary, compact in $\mc{P}(\om)$, and {\em spreading}, that is, whenever $n\in\om\setminus\{0\}$, $\{a_1<a_2<\cdots<a_n\}\in\mc{S}_\al$, $b_1<b_2<\cdots<b_n$, and $a_i\leq b_i$ for every $i$, then
$\{b_1<b_2<\cdots<b_n\}\in\mc{S}_\al$ as well (see \cite[Proposition 4.2]{Alspach-Argyros}).

For $\al<\om_1$ let $X_\al=\mrm{EXH}(\Phi_{\mc{S}_\al})$ be the {\em $\al$th Schreier space}  and $\mc{I}_\al=\mc{I}_{\mc{S}_\al,\lam}$ the {\em $\al$th Schreier  ideal} in the sense of Section \ref{main}, that is,
$\mc{I}_\al=\mathrm{Exh}(\varphi_\al)$, where $\varphi_\al(A) =\varphi_{\mc{S}_\al,\lam}(A)= \Phi_{\mc{S}_\al}(P_A(\lambda))=\sup\{\sum_{i\in A\cap F}\lam_i:F\in\mc{S}_\al\}$. For each $\alpha<\omega_1$, $\mrm{rk}(\mc{S}_\al)=\om^\al+1$, and hence $\mc{S}_\al$ is homeomorphic to $\om^{\om^\al}+1$, furthermore $X_\al$ can be embedded in  $C(\mc{S}_\al)$ (see \cite[Propositions 4.9 and 4.10]{Alspach-Argyros}). 

Clearly, $X_0=c_0$ and $\mc{I}_0=\mc{P}(\om)$.

\begin{prop}
$\mathcal{I}_1 = \mathcal{Z}$.
\end{prop}
\begin{proof}
Let $P_n=[2^n,2^{n+1})$, $n\in\om$. Since $\mc{P}=\{P_n\colon n\in\om\}\subseteq\mc{S}_1$,  $\varphi_1\geq\varphi_{\mc{P},\lam}$ and hence $\mathcal{I}_1 \subseteq \mathcal{Z}=\mrm{Exh}(\varphi_{\mc{P},\lam})$ follows. Conversely, let $A\in \mathcal{Z}$ and denote
	$a_n = |A \cap P_n| /2^n$. Fix an $\eps>0$, assume that $a_n<\eps$ for every $n\geq N$ for some $N$, and also fix an $F\in\mc{S}$. We show that $s\colon =\sum\{\lambda_k\colon k\in F\cap A\setminus 2^N\}< (1-\log_2\eps)\eps$ and hence $\varphi_1(A \setminus 2^N) \leq  (1-\log_2\eps)\eps \xrightarrow{\eps\to 0}0$.

When estimating $s$ from above we can assume that $F\subseteq A\setminus 2^N$ (otherwise we can switch to a ``better'' $F$), in particular (a) $\min (F)\in P_{n_0}$ for some $n_0\geq N$, (b) $|F\cap P_n|<\eps 2^n$ for every $n$, and (c) $s=\sum_{k\in
F}\lambda_k$. If we consider only (a) and (b) above and that $F\in\mc{S}_1$ i.e. $|F|\leq\min(F)$, then $s$ is maximal if $F$ contains the maximal amount of points from consecutive $P_n$s starting with $P_{n_0}$. For such an $F$ we have $s<(n_1-n_0)\eps$ where $n_1$ is the first $n$ such that $F\cap P_n=\0$, this holds e.g. if $n_1$ satisfies $\eps(2^{n_0}+2^{n_0+1}+\cdots+2^{n_1})\geq 2^{n_0+1}>|F|$. A straightforward calculation shows that $n_1= n_0+1-\lfloor\log_2\eps\rfloor$ is large enough, and hence $s< (1-\log_2\eps)\eps$.
\end{proof}

Notice that if $\alpha<\beta$, then \emph{essentially} $\mathcal{S}_\al \subseteq \mathcal{S}_\be$, i.e. there is $n\in \omega$ such that if $S\in \mathcal{S}_\alpha$ and $n<\min S$, then $S\in \mathcal{S}_\beta$ (in notation $\mc{S}_\al\clrest (\om\setminus n)\subseteq\mc{S}_\be$). Using this remark, it
is trivial to see that  $\mathcal{I}_\beta \subseteq\mc{I}_\al$ for $\al<\be$ and so $(\mc{I}_\alpha)_{\alpha<\omega_1}$ forms a non-increasing sequence of ideals. Actually, we know more:

\begin{prop} If $0<\be$ then $\mathcal{I}_\be$ is not $F_\sigma$, and if $\alpha<\beta$, then  $\mathcal{I}_\beta \subsetneq \mathrm{Fin}(\varphi_\beta) \subsetneq \mathcal{I}_\alpha$ (with the only exception $\mrm{Fin}(\varphi_1)=\mc{P}(\om)=\mc{I}_0$). In particular, if $\be$ is a limit ordinal, then $\mc{I}_\be\subsetneq\bigcap_{\al<\be}\mc{I}_\al$.
\end{prop}
\begin{proof}
As $\mc{S}_\be$ is compact and $\mc{I}_\be$ is a non-trivial $\mc{S}_\be$-ideal, it is not $F_\sigma$ by Theorem \ref{compactly-supported}.

\smallskip
Now let $\al<\be$. Then $\mc{I}_\be=\mrm{Exh}(\varphi_\beta)\subseteq\mrm{Fin}(\varphi_\be)$ and as $\mc{I}_\be$ is not $F_\sigma$, this inclusion is strict. Similarly, it is enough to show that $\mrm{Fin}(\varphi_\be)\subseteq\mc{I}_\al$ because $\mc{I}_\al$ is not $F_\sigma$ for $\al>0$. If $A\notin \mathcal{I}_\alpha$ then there are an $\eps>0$ and a sequence $F_0<F_1<\cdots$ in $\mc{S}_\al$ such that $F_n\subseteq A$ and $\varphi_\al(F_n)>\varepsilon$ for every $n$. Now as $n\leq \min (F_n)$, $E_n=F_n\cup F_{n+1}\cup\cdots F_{2n-1}\in\mc{S}_{\al+1}$, and since $\mathcal{S}_{\alpha+1}$ is essentially included in $\mathcal{S}_\beta$, we can assume that $E_n\in\mc{S}_\be$ for every $n$, therefore $\varphi_\be(A)\geq\varphi_\be(E_n)> n\eps$ for every $n$, and so $A\notin \mrm{Fin}(\varphi_\beta)$.
\end{proof}

The ideals $\mathcal{I}_\al$ for $\al>1$ resemble the density zero ideal $\mc{Z}$ and may be considered as kind of density ideals of ``higher order''. However, the following proposition indicates that they are, in a sense, far from being (generalized)
density ideals.
An lsc submeasure $\varphi$ is {\em summable-like} if there is an $\eps>0$ such that for every $\de>0$ there is a sequence $A_n\in [\om]^{<\om}$ of pairwise disjoint finite sets with $\varphi(A_n)<\de$ and there is a $k\in \om$ such that $\varphi\big(\bigcup_{n\in Y}A_n\big)\geq\eps$ for every $Y\in [\om]^k$. An analytic P-ideal $\mc{J}$ is {\em summable-like} if there is a summable-like submeasure $\varphi$ such that $\mc{J}=\mathrm{Exh}(\varphi)$.  For example, Farah ideals which are not trivial modifications of $\mrm{Fin}$ are summable-like, and we already mentioned that so is $\mrm{tr}(\mc{N})$. As far as we know, $\mrm{tr}(\mc{N})$ is the only known ``natural'' non $F_\sigma$ example of such an ideal.

\begin{prop}
$\mathcal{I}_\al$ is summable-like for $\al>1$.
\end{prop}
\begin{proof} Fix $N\in \omega\setminus\{0\}$. We will define a sequence $F_0<F_1<\cdots$ such that $\varphi_\al(F_n)=2^{-N}$ for every $n$ and $\varphi_\al(\bigcup_{n\in H}F_n)= 1$ for every $H\in [\om]^{2^N}$. Let $F_n \subseteq [2^{N+n}, 2^{N+n+1})$ be first $2^n$ many points in this interval. Then  $F_n\in\mc{S}_1\subseteq\mc{S}_\al$ (it is easy to see that $\mc{S}_n\clrest(\om\setminus n)\subseteq\mc{S}_\al$ for every $n\in\om$ and $\al\in [n,\om_1)$) and hence $\varphi_\al(F_n) = 2^{-N}$. Now, if $H\in [\omega]^{2^N}$ then $\bigcup_{n\in H} F_n\in \mc{S}_2\clrest (\om\setminus 2)\subseteq\mc{S}_\al$ (here we need that $\min(F_0)\geq 2$), and therefore $\varphi_\al(\bigcup_{n\in H} F_n) = 1$.
\end{proof}

\begin{rem} So far we know that if $\al>0$, then $\mc{I}_\al$ is not $F_\sigma$ and if $\al>1$ it is summable-like, and of course, all of them contain the summable ideal. Since $\mrm{tr}(\mathcal{N})\subseteq \{A\subseteq 2^{<\om}:|A\cap \,\!^n
	2|/2^n\to 0\} = \mathcal{Z}$, one may wonder if every $\mathcal{I}_\alpha$ contains $\mrm{tr}(\mathcal{N})$. This is not the case, $\mathrm{tr}(\mathcal{N}) \nsubseteq \mathcal{I}_2$: Here of course we identify $\om\setminus\{0\}$ with $2^{<\om}$ in the standard way. Let $A\subseteq\om\setminus\{0\}$ such
that \[ A \cap [2^n, 2^{n+1})=\big\{\text{the first}\;2^{n-k}\;\text{many points in this interval }\big\}\in\mc{S}_1\]  where $n\in [2^k, 2^{k+1})$. Now, $A\in\mrm{tr}(\mc{N})$ because $A\cap [2^n,2^{n+1})$ is a maximal antichain in $A\setminus 2^n$, and hence $\varphi_\mrm{trn}(A\setminus 2^n)=2^{-k}$ if $n\in [2^k,2^{k+1})$. On the other hand, if we fix $k$ then \[ F=\bigcup\big\{A\cap [2^n,2^{n+1})\colon n\in [2^k,2^{k+1})\big\}\in \mc{S}_2\] because it is the union of $2^k$ many consecutive elements of $\mc{S}_1$ and $\min(F)=2^{2^k}$, therefore $\varphi_2(A\setminus 2^{n_0})\geq\sum\{2^{n-k}2^{-n}\colon n\in [2^k,2^{k+1})\}=1$.
\end{rem}

We know that $\mc{I}_\al$ is representable (via an unconditional basic sequence) in $X_\al$ and so in $C(\mc{S}_\al)$ where $\mc{S}_\al$ is compact and of Cantor-Bendixson rank $\om^\al+1$.

\begin{prob} Is it true that $\mc{I}_\al$ cannot be represented (at least via unconditional basic sequences) in any $C(K)$ where $K$ is countable, compact, and of Cantor-Bendixson rank $\leq \om^\al$?
\end{prob}

\section{Schur Property}\label{Schur}

Combining Theorem \ref{characterization} and Theorem \ref{compactly-supported} we can characterize both those spaces $X_\mathcal{F}$ which does not contain a copy of $c_0$ and those which does not contain a copy of $\ell_1$. Both of the
characterizations are topological but of different nature: If $\mathcal{F}$ is a family of finite sets covering $\omega$, then (1) $X_\mathcal{F}$ does not contain a copy of $c_0$ iff $X_\mathcal{F}$ is $F_\sigma$ in $\mbb{R}^\om$; and (2) $X_\mathcal{F}$ does not contain a copy of $\ell_1$ iff $\mathcal{F}$ is precompact.

Since $X_\mathcal{F}$ do contain a copy of $c_0$ if $\mathcal{F}$ is compact, it follows that every $X_\mathcal{F}$ contains either $c_0$ or $\ell_1$. Therefore, the spaces $X_\mathcal{F}$ are in a sense alloys of $c_0$ and $\ell_1$.

We already presented several examples of $X_\mathcal{F}$ and $\mathcal{F}$-ideals for $\mathcal{F}$ being compact. In this section we want to study the families which are far from being compact and which induce Banach spaces from the other extreme.

First, we will show that a family motivated by  Farah's example induces an $\ell_1$-saturated Banach space which is not isomorphic to $\ell_1$.

\begin{exa}\label{Farah-family}  Let $\mathcal{F}$ be the family of those finite sets $F\subseteq\om\setminus\{0\}$ for which $|F\cap [2^n, 2^{n+1})|\leq n^{-1}2^{n}$ for every $n>0$. Then $\mathcal{I}_{\mathcal{F},\lam}$ is a variant of Farah example (see Example \ref{exa-farah}).

Concerning the generated Banach spaces it is not hard to see that $X_\mathcal{F}=\mrm{FIN}(\Phi_\mathcal{F})$ and so $X_\mathcal{F}$ does not contain a copy of $c_0$. In fact, we can prove that $X_\mathcal{F}$ has a slightly stronger property. Recall that a Banach space $X$ satisfies the \emph{Schur property} if every weakly null sequence in $X$ converges to $0$ in norm. The canonical space with Schur property is $\ell_1$ and Rosenthal $\ell_1$ theorem (see Theorem \ref{Ros}) implies that every space with Schur property is $\ell_1$-saturated.

We will show that $X_\mathcal{F}$ enjoys the Schur property. Indeed, assume that $(x_n)$ is such that $\Phi(x_n)>\varepsilon$ for some $\varepsilon>0$ and each $n$. We are going to show that $(x_n)$ is not weakly null. Since finitely dimensional
spaces enjoy Schur property, without loss of generality we may assume that $(x_n)$ is a block sequence, i.e. there is a sequence of finite sets $(G_n)$ such that $\mathrm{supp}(x_n)\subseteq G_n$ and $G_n < G_{n+1}$ for every $n$. Moreover, passing to a subsequence if needed, we may assume that $|\{k\colon [2^n,2^{n+1})\cap G_k \ne \emptyset\}| \leq 1$ for every $n$. Let $A_n = \mathrm{supp}(x_k) \cap
		[2^n, 2^{n+1})$ if $[2^n,2^{n+1})\cap  G_k\ne \emptyset$. If there is no such $k$, let $A_n = \emptyset$. We may assume that $|A_n|<n^{-1}2^n$ for every $n$ (by shrinking supports of $x_k$'s if needed, without decreasing their norms).

Now, notice that $\bigcup_{n\in\om} A_n$ is infinite and $\bigcup_{n\in\om} A_n \in \overline{\mathcal{F}}$. This means that the function $f\colon X_\mathcal{F} \to \mathbb{R}$ defined by $f(x) = \sum_{n\in \omega} \chi_A(n) x(n)$ is a functional on $X_\mathcal{F}$ of norm $1$ such that $f(x_n)>\varepsilon$ for every $n$, and so $(x_n)$ is not weakly null.
		
On the other hand $X_\mathcal{F}$ is not isomorphic to $\ell_1$. Otherwise, $\{e_n\colon n\in \omega\}$ would be equivalent to the standard base of $\ell_1$ (as all bounded unconditional bases of $\ell_1$ are equivalent, see \cite{Lindenstrauss}). But there is no $c>0$ such that $\Phi_\mc{F}(\sum_{n<N} e_n)>c\cdot N  = \|\sum_{n<N} e_n\|_1$ for every $N\in \omega$.
\end{exa}

There are many examples having the Schur property which are not isomorphic to $\ell_1$. It would be more interesting to find a space which is $\ell_1$-saturated and which does not have the Schur property. There are many examples of them, too (the
first one was given by Bourgan, \cite{Bourgain}, see also the construction of Azimi and Hagler, \cite{Hagler}), but most of these examples are quite technical. However, it is not clear if there is a space $X_\mathcal{F}$ which does not have a
copy of $c_0$ and does not posses the Schur property.

\begin{prob}\label{schur-c0} Is there a family $\mathcal{F}\subseteq [\omega]^{<\omega}$ such that $X_\mathcal{F}$ does not have the Schur property and which does not have a copy of $c_0$?
\end{prob}

The following remarks (motivated by the example above) may indicate how to look for such a family (or how to prove that they do not exist).

\begin{prop}\label{clos} Let $\mathcal{F}$ be a hereditary family of finite sets covering $\om$. Assume that $(x_n)$ is a sequence in $X_\mathcal{F}$ with pairwise disjoint supports and that there is $A \in \overline{\mathcal{F}}$ such that $\limsup_{n\to\infty} \Phi_\mathcal{F}(P_A(x_n)) > 0$. Then $(x_n)$ is not weakly null.
\end{prop}

\begin{proof}
Let $(x_n)$ and $A\subseteq \omega$ be as above. Passing to a subsequence, we can assume that there is an $\eps>0$ such that $\Phi_\mathcal{F}(P_A(x_n))>\varepsilon$ for every $n$.
	
Let $F_n = \mrm{supp}(x_n)$, $F^+_n = \{k\in F_n\colon x_n(k)>0\}$, and $F^-_n = F_n \setminus F^+_n$. Notice that either $\Phi_\mathcal{F}(P_{F^+_n}(x_n))>\varepsilon/2$ or $\Phi_\mathcal{F}(P_{F^-_n}(x_n))>\varepsilon/2$. Hence, shrinking $A$ again if needed, we may assume that $|\sum_{k\in F_n\cap A} x_n(k)|>\varepsilon/2$ for every $n$.

Now, let $x^*\colon X_\mathcal{F} \to \mathbb{R}$, $x^*(x) = \sum_{n\in A} x(n)$. Since $A\in \overline{\mathcal{F}}$, $x^*\in X^*_\mathcal{F}$ (in fact, $x^*$ is of norm $1$) but $|x^*(x_n)| = |\sum_{k\in F_n\cap A} x_n(k)| > \varepsilon/2$ for every $n$, and hence $(x_n)$ is not weakly null.
\end{proof}

\begin{cor}\label{szczur} Let $\mathcal{F}$ be as above. If for  every sequence $(x_n)$ with pairwise disjoint supports which is not (strongly) null, there is $A\in \overline{\mathcal{F}}$ such that $\limsup_{n\to\infty} \Phi_\mathcal{F}(P_A(x_n))>0$, then $X_\mathcal{F}$ has the Schur property. 
\end{cor}

\begin{rem}
If a sequence $(x_n)$ in $X_\mathcal{F}$ has $A\in \overline{\mathcal{F}}$ as above, then it has a subsequence which is equivalent to the standard basis of $\ell_1$.
\end{rem}

\begin{prob} Let $\mathcal{F}$ be as above. Assume that there is a normalized sequence $(x_n)$ in $X_\mathcal{F}$  such that for each $\lim_{n\to\infty} \Phi_\mathcal{F}(P_A(x_n)) = 0$ for each $A\in \overline{\mathcal{F}}$. Does it mean that $X_\mathcal{F}$ contains a copy of $c_0$? Or, at least, that it is not $\ell_1$-saturated?
\end{prob}

Notice that if the answer to the above is positive, then, because of Corollary \ref{szczur}, Problem \ref{schur-c0} has the negative answer.

\begin{prob}
	Assume $\mathcal{F} \subseteq [\omega]^{<\omega}$ is a family such that $X_\mathcal{F}$ does not have a copy of $c_0$. Is $X_\mathcal{F}$ $\ell_1$-saturated?
\end{prob}

This problem is natural in the light Theorem \ref{compactly-supported}: We know that $X_\mathcal{F}$ does not have a copy of $\ell_1$ if and only if it is $c_0$-saturated. The above problem asks if we have a similar
equivalence for spaces which do not have a copy of $c_0$. It seems likely that the answer is positive.

\bibliographystyle{alpha}
\bibliography{bib-norm}
\end{document}